\documentclass[11pt,letterpaper,reqno]{amsart}
\usepackage{fullpage}
\usepackage[top=2cm, bottom=4.5cm, left=2.5cm, right=2.5cm]{geometry}
\usepackage{amsmath,amsthm,amsfonts,amssymb,amscd}
\usepackage{geometry,tikz,tikz-cd}
\usepackage{empheq}
\usepackage{lipsum}
\usepackage{lastpage}
\usepackage{thmtools}
\usepackage{enumerate}
\usepackage[shortlabels]{enumitem}
\usepackage{fancyhdr}
\usepackage{mathrsfs}
\usepackage{xcolor}
\usepackage{graphicx}
\usepackage{listings}
\usepackage{bbm}
\usepackage{array}
\usepackage{hyperref}
\usepackage[makeroom]{cancel}

\usepackage[utf8]{inputenc}
\usepackage{comment}

\usepackage{makecell}

\usepackage[all]{xy}
\xyoption{matrix}
\xyoption{arrow}

\usetikzlibrary{arrows,decorations.pathmorphing,backgrounds,positioning,fit,petri}

\tikzset{help lines/.style={step=#1cm,very thin, color=gray},
help lines/.default=.5} 
\tikzset{thick grid/.style={step=#1cm,thick, color=gray},
thick grid/.default=1} 

\mathtoolsset{showonlyrefs=true}
\numberwithin{figure}{section}
\numberwithin{table}{section}

\theoremstyle{definition}

\theoremstyle{plain}
\newcommand{\thistheoremname}{}
\newtheorem*{genericthm*}{\thistheoremname}
\newenvironment{namedthm*}[1]
  {\renewcommand{\thistheoremname}{#1}%
   \begin{genericthm*}}
  {\end{genericthm*}}

\hypersetup{%
  colorlinks=true,
  linkcolor=blue,
  citecolor=blue,
  linkbordercolor={0 0 1}
}

\DeclareMathOperator{\ZZ}{\mathbb{Z}}

\newcommand{\condf}{{\mathfrak{f}}}

\newcommand{\lrabs}[1]{\left\lvert #1 \right\lvert}
\newcommand{\lrp}[1]{\left(#1\right)}
\newcommand{\lrb}[1]{\left[#1\right]}

\allowdisplaybreaks

\newtheorem{theorem}{Theorem}[section]
\newtheorem{corollary}[theorem]{Corollary}
\newtheorem{lemma}[theorem]{Lemma}
\newtheorem{proposition}[theorem]{Proposition}
\theoremstyle{definition}

\theoremstyle{definition}

\theoremstyle{remark}
\newtheorem{conjecture}[theorem]{\bf Conjecture}

\numberwithin{equation}{section}

\usepackage{latexsym}
\usepackage{amssymb}

\newcommand{\ben}{\begin{equation}}
\newcommand{\een}{\end{equation}}

\newcommand{\SL}[1]{\ensuremath{{\mathrm {SL}_{ #1 }}}}

\makeatletter
\NewDocumentCommand{\sump}{e{_}}
 {%
  \DOTSB
  \mathop{\IfNoValueTF{#1}{\sump@{}}{\sump@{#1}}}%
  \nolimits
 }
\newcommand{\sump@}[1]{\mathpalette\sump@@{#1}}
\newcommand{\sump@@}[2]{%
  \ifx#1\displaystyle
    {\sump@display{#2}}%
  \else
    \sum@\nolimits'_{#2}%
  \fi
}
\newcommand{\sump@display}[1]{%
  \sbox\z@{$\m@th\displaystyle\sum@\nolimits'$}%
  \sbox\tw@{$\m@th\displaystyle\sum@\limits_{#1}$}%
  \sbox\@tempboxa{$\m@th\displaystyle'$}
  \mathop{\sum@\nolimits' \kern-\wd\@tempboxa}\limits_{#1}%
  \ifdim\wd\z@>\wd\tw@
    \kern\dimexpr\wd\z@-\wd\tw@\relax
  \fi
}
\makeatother

\newcommand{\nw}{{\textnormal{new}}}

\DeclareMathOperator{\old}{old}
\DeclareMathOperator{\Tr}{Tr}

\DeclareMathOperator{\lcm}{lcm}

\setlist[enumerate]{leftmargin=*,widest=0}
\setlist[itemize]{leftmargin=*,widest=0}
\makeatletter
\def\subsection{\@startsection{subsection}{2}%
  \z@{.5\linespacing\@plus.7\linespacing}{.3\linespacing}%
  {\normalfont\bfseries}}

\def\subsubsection{\@startsection{subsubsection}{3}%
  \z@{.5\linespacing\@plus.7\linespacing}{.3\linespacing}%
  {\normalfont\bfseries}}
\makeatother

\title{Nonvanishing of second coefficients of Hecke polynomials on the newspace}


\author[W. Cason]{William Cason}
\address[W. Cason]{School of Mathematical and Statistical Sciences, Clemson University, Clemson, SC, 29634}
\email{wbcason@clemson.edu}

\author[A. Jim]{Akash Jim}
\address[A. Jim]{Department of Mathematics, Princeton University, Princeton, NJ, 08540}
\email{ajim@princeton.edu}

\author[C. Medlock]{Charlie Medlock}
\address[C. Medlock]{School of Mathematical and Statistical Sciences, Clemson University, Clemson, SC, 29634}
\email{medloc5@clemson.edu}

\author[E. Ross]{Erick Ross}
\address[E. Ross]{School of Mathematical and Statistical Sciences, Clemson University, Clemson, SC, 29634}
\email{erickr@clemson.edu}

\author[T. Vilardi]{Trevor Vilardi}
\address[T. Vilardi]{School of Mathematical and Statistical Sciences, Clemson University, Clemson, SC, 29634}
\email{tvilard@clemson.edu}

\author[H. Xue]{Hui Xue}
\address[H. Xue]{School of Mathematical and Statistical Sciences, Clemson University, Clemson, SC, 29634}
\email{huixue@clemson.edu}

\begin{document}

\begin{abstract}
For $m \geq 1$, let $N \geq 1$ be coprime to $m$, $k \geq 2$, and $\chi$ be a Dirichlet character modulo $N$ with $\chi(-1)=(-1)^k$. Then let $T_m^{\text{new}}(N,k,\chi)$ denote the restriction of the $m$-th Hecke operator to the space $S_k^{\text{new}}(\Gamma_0(N), \chi)$. We demonstrate that for fixed $m$ and trivial character $\chi$, the second coefficient of the characteristic polynomial of $T_m^{\text{new}}(N,k)$ vanishes for only finitely many pairs $(N,k)$, and we further determine the sign. To demonstrate our method, for $m=2,4$, we also compute all pairs $(N,k)$ for which the second coefficient vanishes. In the general character case, we also show that excluding an infinite family where $S_k^{\text{new}}(\Gamma_0(N), \chi)$ is trivial, the second coefficient of the characteristic polynomial of $T_m^{\text{new}}(N,k,\chi)$ vanishes for only finitely many triples $(N,k,\chi)$.
\end{abstract}

\maketitle

\section{Introduction}

Let $S_k(\Gamma_0(N), \chi)$ denote the space of cusp forms of weight $k\geq2$, level $N\geq1$, and character $\chi$. Here, $\chi$ is a Dirichlet character modulo $N$ such that $\chi(-1)=(-1)^k$. For $m \geq 1$ coprime to $N$, we will denote the $m$-th Hecke operator on $S_k(\Gamma_0(N), \chi)$ by $T_m(N,k,\chi)$. (Throughout this entire we will only consider $m$ coprime to $N$.) When the character $\chi$ is trivial, we will drop $\chi$ and simply write  $S_k(\Gamma_0(N))$ and $T_m(N,k)$, respectively. 

The space of cusp forms has a decomposition  $S_k(\Gamma_0(N), \chi)=S_k^{\old}(\Gamma_0(N), \chi)\oplus S_k^\nw(\Gamma_0(N), \chi)$; see Cohen and Stromberg \cite[Proposition~13.3.2]{cohen-stromberg}. The subspaces in this decomposition are orthogonal complements with respect to the Petersson inner product, and further, they are stable under the Hecke operator $T_m(N,k,\chi)$.
We write $T_m^\nw(N,k,\chi)$ for the restriction of $T_m(N,k,\chi)$ to the new subspace $S_k^\nw(\Gamma_0(N), \chi)$.

Let $d=\dim S_k(\Gamma_0(N),\chi)$ and $n = \dim S_k^\nw(\Gamma_0(N), \chi)$. We write the characteristic polynomials of $T_m(N,k,\chi)$ and $T_m^\nw(N,k,\chi)$ as
\begin{align}
T_m(N,k,\chi)(x)&=\sum_{i=0}^d (-1)^i a_i(m,N,k,\chi) x^{d-i},\quad \text{and}\\
    T_m^\nw(N,k,\chi)(x)
    &=\sum_{i=0}^n (-1)^i a_i^\nw(m,N,k,\chi) x^{n-i}, 
\end{align}
respectively. Here, we refer to $a_i(m,N,k,\chi)$ and $a_i^\nw(m,N,k,\chi)$ as the $i$-th coefficient of the Hecke polynomials for $T_m(N,k,\chi)$ and $T_m^\nw(N,k,\chi)$, respectively. Again, if $\chi$ is trivial, we will drop it from the notation. Observe that the first coefficients $a_1(m,N,k,\chi)$ and $a_1^\nw(m,N,k,\chi)$ are the traces of $T_m(N,k,\chi)$ and $T_m^\nw(N,k,\chi)$, respectively. 

In \cite{rouse}, Rouse conjectured that the traces of Hecke operators (that is $a_1(m,N,k)$) are nonvanishing for non-square $m$, $N$ coprime to $m$, and even $k\geq 16$ or $k=12$. This is a generalization of Lehmer's conjecture \cite{Lehmerconjecture} on the nonvanishing of Ramanujan's $\tau$ function. The work of Clayton et al. \cite{clayton-et-al} and Ross and Xue \cite{ross-xue} studied a related question: the nonvanishing of the second coefficients $a_2(m,N,k,\chi)$. 
In this paper, we study the the nonvanishing of the second coefficients $a_2^\nw(m,N,k,\chi)$ on the new subspace. This new subspace $S_k^\nw(\Gamma_0(N), \chi)$ deserves separate attention; it is generated by newforms, and hence understanding the Hecke operators on $S_k^\nw(\Gamma_0(N), \chi)$ will reveal useful information about newforms. 
The works \cite{clayton-et-al} and \cite{ross-xue} showed that the second coefficient is more predictable than the trace in a certain sense. In this paper, we provide more evidence in this aspect. In particular, we show nonvanishing results for $a_2^\nw(m,N,k,\chi)$ in Theorems \ref{thm:main-theorem} and \ref{thm:m=2-4-Theorem}; however, the analogs of these results have not yet been established for $a_1^\nw(m,N,k,\chi)$.

Our first main result concerns the nonvanishing of $a_2^\nw(m,N,k,\chi)$, and furthermore the sign of $a_2^\nw(m,N,k)$.
Observe that when $\dim S_k^\nw(\Gamma_0(N), \chi) < 2$, the Hecke polynomial for $T_m^\nw(N,k,\chi)$ has degree $< 2$, and hence $a_2^\nw(m,N,k,\chi)$ trivially vanishes.
In the first part of Theorem \ref{thm:main-theorem}, we do not need to exclude the case when $a_2^\nw(m,N,k)$ trivially vanishes, because $\dim S_k^\nw(\Gamma_0(N)) < 2$ for only finitely many pairs $(N,k)$ \cite[Theorem 1.3]{ross}.
However, we do need to exclude the case where $a_2^\nw(m,N,k,\chi)$ trivially vanishes, because as discussed in Section \ref{sec:character-case}, $\dim S_k^\nw(\Gamma_0(N), \chi) = 0$ for an infinite family of triples $(N,k,\chi)$.

\begin{theorem} \label{thm:main-theorem}
Let $m\ge1$ be fixed. Consider $N \geq 1$ coprime to $m$ and $k \geq 2$ even. Then for all but finitely many pairs $(N,k)$,
$$
a_2^\nw(m, N, k) \text{ is  } \begin{cases}
   \text{negative} & \text{when } m \text{ is not a perfect square,} \\
   \text{positive} & \text{when } m \text{ is a perfect square.}
\end{cases}
$$
Furthermore, consider $N \geq 1$ coprime to $m$, $k \geq 2$, and $\chi$ a Dirichlet character modulo $N$ where $\chi(-1) = (-1)^k$.  Then $a_2^\nw(m,N,k,\chi)$ nontrivially vanishes for only finitely many triples $(N,k,\chi)$.
\end{theorem}

The approach we use here is similar to that adopted in \cite{clayton-et-al} and \cite{ross-xue}. We first express $a_2^\nw(m, N, k)$ in terms of traces of various Hecke operators. Then for fixed $m$, the Eichler-Selberg trace formula is applied to determine the asymptotic behavior of $a_2^\nw(m, N, k)$ with respect to $N$ and $k$.

For our second main result, we compute explicit bounds for each of the terms in the Eichler-Selberg trace formula and use these bounds to determine the exceptional pairs for the cases of $m = 2$ and $m = 4$ with trivial character. 
Here, we give the pairs $(N,k)$ for which $a_2^\nw(m,N,k)$ nontrivially vanishes. (The pairs $(N,k)$ for which $\dim S_k^\nw(\Gamma_0(N)) < 2$ and hence $a_2^\nw(m,N,k)$ trivially vanishes can be found in Ross \cite[Tables 6.2, 6.3]{ross}.) 

\begin{theorem} \label{thm:m=2-4-Theorem}
    Consider $N\geq1$ odd and $k\geq2$ even. Then $a_2^\nw(2,N,k)$ nontrivially vanishes precisely for $(N,k) \in \{ (37,2),\,  (57,2) \}$, and $a_2^\nw(4,N,k)$ nontrivially vanishes precisely for $(N,k) \in \{ (43,2), (57,2) ,(75,2) ,(205,2)\}$.
\end{theorem}

The paper is organized as follows. In Section \ref{sec:coefficient-formula}, we express $a_2^\nw(m, N, k, \chi)$ in terms of traces of Hecke operators and state the Eichler-Selberg trace formula for the new subspace. Then in Section \ref{sec:Ai-lincomb-multiplicative}, we bound each of the terms appearing in the Eichler-Selberg trace formula for $\Tr T_m^\nw(N, k,\chi)$. In Section \ref{sec:main-thm}, we prove the trivial character case of Theorem \ref{thm:main-theorem} using the bounds obtained in Section \ref{sec:Ai-lincomb-multiplicative}. In Section \ref{sec:T2}, we prove Theorem \ref{thm:m=2-4-Theorem}, determining the complete list of pairs $(N, k)$ for which $a_2^\nw(2, N, k)$ and $a_2^\nw(4, N, k)$  vanish. In Section \ref{sec:character-case}, we prove the general character case of Theorem \ref{thm:main-theorem}.
Finally, in Section \ref{sec:discussion}, we discuss these results, giving some additional motivation and surveying potential future work.

\section{The Second Coefficient Formula} \label{sec:coefficient-formula}

In the manner of \cite[Proposition 2.1]{clayton-et-al} and \cite[Lemma 2.1]{ross-xue}, we develop a formula for $a_2^\nw(m,N,k,\chi)$ in terms of traces of Hecke operators.
\begin{lemma} \label{lem:a2-coeff-formula}
    For convenience, let $T_m^{\textnormal{new}}$ denote $T_m^{\nw}(N,k,\chi)$. Then 
    \begin{align}     
    a_2^{\nw}(m,N,k,\chi) = \frac{1}{2} \lrb{ \left(\Tr T_m^{\nw}\right)^2 - \sum_{d \mid m} \chi(d) d^{k-1} \Tr T_{m^2/d^2}^{\nw} }. 
    \end{align}
\end{lemma}
\begin{proof}
    Let $\lambda_1, \ldots , \lambda_n$ be the eigenvalues of $T_m^\nw$. Then by the definition of the characteristic polynomial and the Hecke operator composition formula \cite[Theorem 10.2.9]{cohen-stromberg},  we have 
    \begin{align}
        a_2^\nw(m,N,k,\chi)
        &= \sum_{1 \leq i < j \leq n} \lambda_i \lambda_j \\
        &= \frac12 \lrb{ \lrp{ \sum_{1 \leq i \leq n} \lambda_i }^2 - \sum_{1 \leq i \leq n} \lambda_i^2 } \\
        &= \frac12 \lrb{ (\Tr T_m^\nw)^2 - \Tr (T_m^\nw)^2 } \\
        &= \frac12 \lrb{ (\Tr T_m^\nw)^2 - \sum_{d \mid m} \chi(d) d^{k-1} \Tr T_{m^2/d^2}^\nw },
    \end{align} 
as desired.
\end{proof}

We now state the Eichler-Selberg trace formula. 
\begin{lemma}[{\cite[pp.~370-371]{knightly-li}, \cite[Theorem 12.4.11]{cohen-stromberg}}]

\label{lemma:eichler-selberg-trace-formula}
Let $m \geq 1$, $N \geq 1$, $k \geq 2$, and $\chi$ be a Dirichlet character modulo $N$ such that $\chi(-1) = (-1)^k$.  
Then
\begin{equation} \label{eqn:eichler-selberg-trace-formula}
\Tr T_m(N,k,\chi) = A_1(m,N,k,\chi) - A_2(m,N,k,\chi) - A_3(m,N,k,\chi) + A_4(m,N,k,\chi)
\end{equation}
where
\begin{align}
    \label{eqn:A1m-formula}
    A_1(m,N,k,\chi) &= \chi(\sqrt{m}) \frac{k-1}{12} \psi (N) m^{k/2-1},  \\
   \label{eqn:A2m-formula}
   A_2(m,N,k,\chi) &= \frac{1}{2} \sum_{t^2 < 4m} U_{k-1}(t,m) \sum_{n}  h_w \lrp{ \frac{t^2-4m}{n^2} } \mu_{t,n,m}(N), \\
   \label{eqn:A3m-formula}
   A_3(m,N,k,\chi) &= \frac{1}{2} \sum_{d|m} \min(d,m/d)^{k-1} \sum_{\tau} \phi ( \gcd(\tau, N / \tau )) \chi(y_{\tau}), \\
   \label{eqn:A4m-formula}
   A_4(m,N,k,\chi) &= \begin{dcases} 
      \sum_{\substack{c|m \\ (N, m/c) = 1}} c & \text{if $k=2$ and $\chi = \chi_0$},\\
      0 & \text{if $k>2$ or $\chi \neq \chi_0$} .
   \end{dcases}
\end{align}
Here, we have the following notation.
\begin{itemize}
  \item $\chi(\sqrt{m})$ is interpreted as $0$ if $m$ is not a perfect square. 
  
  \item $\psi (N) = \lrb{\, \SL{2}(\ZZ) \, : \, \Gamma_0 (N) \,} = N \prod_{p|N} \lrp{ 1 + \frac{1}{p} }.$ 
  
  \item The outer summation for $A_2(m,N,k,\chi)$ runs over all $t \in \ZZ$ such that $t^2 < 4m$. Note that the terms corresponding to $t=t_0$ and $t=-t_0$ coincide.

  \item $U_{k-1}(t,m)$ denotes the Lucas sequence of the first kind. In particular, $U_{k-1}(t,m) = \frac{\rho^{k-1} - \bar{\rho}^{k-1}}{\rho - \bar{\rho}}$ where $\rho, \bar{\rho}$ are the two roots of the polynomial $x^2-tx+m$.
  
  \item The inner summation for $A_2(m,N,k,\chi)$ runs over all positive integers $n$ such that $n^2 \, | \, (t^2 -4m)$ and $\frac{t^2-4m}{n^2} \equiv 0,1 \pmod{4}$. 

  \item $h_w \lrp{ \frac{t^2-4m}{n^2} }$ is the weighted class number of the imaginary quadratic order with discriminant $\frac{t^2-4m}{n^2}$. This is the usual class number divided by $2$ or $3$ if the discriminant is $-4$ or $-3$, respectively. These are given explicitly in Table \ref{table:weighted-class-numbers} below. 

  \item $\displaystyle \mu_{t,n,m}(N) = \frac{\psi (N)}{\psi (N/\gcd(N,n))} \sideset{}{'} \sum_{c \!\! \mod N} \chi(c)$, where the primed summation runs over all elements $c$ of $(\ZZ / N \ZZ)^{\times}$ which lift to solutions of $x^2-tx+m \equiv 0 \pmod{N \cdot \gcd(N,n)}$. Note that $\mu_{t,n,m}(N)$ is a multiplicative function of $N$ \cite[Proposition 26.41]{knightly-li}.

  \item The outer summation for $A_3(m,N,k,\chi)$ runs over all positive divisors $d$ of $m$. Note that the terms corresponding to $d=d_0$ and $d=m/d_0$ coincide.
  
  \item The inner summation for $A_3(m,N,k,\chi)$ runs over all positive divisors $\tau$ of $N$ such that $\gcd(\tau, N/\tau)$ divides $\gcd(N/\condf(\chi), d-m/d)$. Here, $\condf(\chi)$ is the conductor of $\chi$.

  \item $\phi$ is the Euler totient function.

  \item $y_{\tau}$ is the unique integer modulo $\lcm(\tau, N/\tau)$ determined by the congruences $y_{\tau} \equiv d \pmod{\tau}$ and $y_{\tau} \equiv \frac{m}{d} \pmod{\frac{N}{\tau}}$.

  \item $\chi_0$ denotes the trivial character modulo $N$.

  \item Throughout, $\chi$ is a character modulo $N$, so $\chi(a) = 0$ if $\gcd(a,N) > 1$, even in the trivial character case.
  
\end{itemize}
\end{lemma}

We also give a table of the weighted class numbers $h_w(n)$ used in the Eichler-Selberg trace formula. We will need these values later in Section \ref{sec:T2} when we compute explicit bounds for the various trace formula terms. 
\begin{table}[h!] 
\begin{equation}
\begin{array}{|c|c|c|c|c|c|c|c|c|c|c|c|}
\hline
 n & -3 & -4 & -7 & -8 & -11 & -12 & -15 & -16 & -19 & -20 & -23 \\
\hline
 h_w(n) & \frac{1}{3} & \frac{1}{2} & 1 & 1 & 1 & 1 & 2 & 1 & 1 & 2 & 3 \\
 \hline \hline
 n & -24 & -27 & -28 & -31 & -32 & -35 & -36 & -39 & -40 & -43 & -44 \\
 \hline
 h_w(n) & 2 & 1 & 1 & 3 & 2 & 2 & 2 & 4 & 2 & 1 & 3 \\   
 \hline \hline 
 n & -47 & -48 & -51 & -52 & -55 & -56 & -59 & -60 & -63 & -64 & -67 \\
 \hline
 h_w(n) & 5 & 2 & 2 & 2 & 4 & 4 & 3 & 2 & 4 & 2 & 1 \\ 
 \hline
\end{array}
\end{equation}
\caption{The weighted class numbers $h_w(n)$ \cite[p.~345]{knightly-li}, \cite[A014600]{oeis}.}
\label{table:weighted-class-numbers}
\end{table}

Next, we similarly have the following trace formula for $T_m^\nw(N,k,\chi)$.

\begin{lemma}[{\cite[Theorem 13.5.7 for $\gcd(m,N)=1$]{cohen-stromberg}}] \label{lemma:beta}
Let $m \geq 1, N \geq 1$ be coprime to $m$, $k \geq 2$, and $\chi$ be a Dirichlet character modulo $N$ such that $\chi(-1) = (-1)^k$. Recall that $\condf(\chi)$ denotes the conductor of $\chi$. Let $\beta$ be the multiplicative function defined on prime powers $p^r$ by 
\begin{equation}
    \beta(p^r) = 
    \begin{dcases}
        -2 & \text{if } r = 1, \\
        1 & \text{if } r = 2, \\
        0 & \text{if } r \geq 3.
    \end{dcases}
\end{equation}
Then
\begin{equation} \label{eqn:cohen-stromberg-new-trace-formula}
    \Tr T_m^{\nw}(N,k,\chi) = \sum_{\condf(\chi) \mid M \mid N} \beta\lrp{\frac{N}{M}} \cdot \Tr T_m(M,k,\chi).
\end{equation}
\end{lemma}
We will use this formula to study the second coefficient $a_2^\nw(m,N,k,\chi)$. In Sections \ref{sec:Ai-lincomb-multiplicative} - \ref{sec:T2}, we restrict to the case of trivial character. Then in Section \ref{sec:character-case}, we will extend our arguments to the case of general character. In the case of trivial character, we can reduce \eqref{eqn:cohen-stromberg-new-trace-formula} to
\begin{equation} \label{eqn:reduced-new-trace-formula}
    \Tr T_m^\nw(N,k) = \sum_{M \mid N} \beta\lrp{\frac{N}{M}} \cdot \Tr T_m(M,k).
\end{equation}
Following the notation of Serre \cite{serre}, we apply \eqref{eqn:reduced-new-trace-formula} to the Eichler-Selberg trace formula and write 
\begin{equation} \label{eqn:new-eichler-selberg-trace-formula}
    \Tr T_m^\nw(N,k) = A_1^\nw(m,N,k) - A_2^\nw(m,N,k) - A_3^\nw(m,N,k) + A_4^\nw(m,N,k),
\end{equation}
where
\begin{equation} \label{eqn:Ainew-equals-beta-convolution-Ai}
    A_i^\nw(m,N,k) = \sum_{M \mid N}\beta\lrp{\frac{N}{M}} \cdot A_i(m, M,k).
\end{equation}

\section{Bounding the \texorpdfstring{$A_i^{\nw}(m,N,k)$}{Ainew(m,N,k)}}

\label{sec:Ai-lincomb-multiplicative}

In this section, we write each $A_{i}(m,N,k)$ term from the Eichler-Selberg trace formula as a linear combination of multiplicative functions $f(N)$. This will allow us to rewrite \eqref{eqn:Ainew-equals-beta-convolution-Ai} as a linear combination of Dirichlet convolutions of the form $\beta * f$. We then use these convolutions to give explicit bounds for each of the $A_{i}^\nw(m,N,k)$ terms from \eqref{eqn:new-eichler-selberg-trace-formula}. We also give the asymptotic behavior of these terms. This asymptotic behavior will be stated using big-$O$ notation with respect to $N$ and $k$.

The Dirichlet convolution 
\begin{align}
\beta * f(N) = \sum_{M \mid N} \beta\lrp{\frac{N}{M}} \cdot f(M)
\end{align}
can be computed by the following formula.
\begin{lemma} \label{lemma:dirichlet-convolution}
    Let $f$ be a multiplicative function (not identically zero) and $\beta * f$ denote the Dirichlet convolution of $\beta$ with $f$. Then $\beta * f$ is the multiplicative function defined on prime powers by 
    $$ (\beta * f) (p^r) = \begin{dcases}
        f(p)-2, & \text{if } r=1, \\
        f(p^{r})-2f(p^{r-1})+f(p^{r-2}) & \text{if } r \geq 2.
    \end{dcases} $$
\end{lemma}
This formula follows directly from the definition of $\beta$ given in Lemma \ref{lemma:beta}.


\subsection{Bounding \texorpdfstring{$A_1^\nw(m,N,k)$}{A1new(m,N,k)}}

Recall that 
\begin{equation} \label{eqn:A1-trivial-char}
    A_1(m,N,k) = \chi_0(\sqrt{m}) \frac{k-1}{12} \psi (N) m^{k/2-1},
\end{equation}
where
\begin{equation} \label{eqn:psi-formula}
    \psi(N) = N\prod_{p \mid N}\lrp{1+\frac{1}{p}}.
\end{equation}
Observe that as a function of $N$, $A_1(m,N,k)$ is a multiple of the multiplicative function $\psi(N)$. We now use Lemma \ref{lemma:dirichlet-convolution} to give a lower bound on the convolution $\beta * \psi$.
\begin{lemma} \label{lem:psi-new}
    Let $\psi^\nw := \beta * \psi$. 
    Then
    \begin{align}
        \psi^\nw(N) \geq \frac{N}{\pi_1(N)},
    \end{align}
    where
    \begin{equation} \label{eqn:pi-3-def}
        \pi_1(N) = \prod_{p \mid N} \lrp{1 + \frac{p+1}{p^2-p-1}}.
    \end{equation}
\end{lemma}
\begin{proof}
    Let $p\mid N$ be prime. Applying Lemma \ref{lemma:dirichlet-convolution} to $\psi^\nw$ yields
    \begin{align}
        \psi^\nw(p) &= \psi(p) -2 = p -1, \\
        \psi^\nw(p^2) &= \psi(p^2) -2\psi(p) + 1 = p^2 -p -1 \\
        \psi^\nw(p^r) &= \psi(p^r) -2\psi(p^{r-1}) + \psi(p^{r-2}) \\ 
        &= \lrp{p^r-2p^{r-1}+p^{r-2}} \lrp{1+\frac1p} \\ 
        &= p^r -p^{r-1} -p^{r-2} + p^{r-3} \qquad \text{for } r \geq 3.
    \end{align}
Observe that in each of these three cases, 
\begin{align}
    \psi^\nw(p^r) &\geq p^r \lrp{1 - \frac1p - \frac{1}{p^2}} = \frac{p^r}{\lrp{\frac{p^2}{p^2-p-1}}} = \frac{p^r}{\lrp{1 + \frac{p+1}{p^2-p-1}}},
\end{align} 
verifying the desired result. 
\end{proof}

We can then employ \eqref{eqn:A1-trivial-char} to write $A_1^\nw(m,N,k)$ as
\begin{align}
    A_1^\nw (m,N,k) &= \sum_{M \mid N} \beta\lrp{\frac{N}{M}} \cdot A_1 (m,M,k) \\
    &= \chi_0(\sqrt{m}) \frac{k-1}{12} \psi^\nw(N) m^{k/2-1}. \label{eqn:A1m-new-formula}
\end{align}


\subsection{Bounding \texorpdfstring{$A_2^\nw(m,N,k)$}{A2new(m,N,k)}}

Next, recall that
\begin{equation}
    A_{2}(m,N,k) = \frac{1}{2} \sum_{t^2 < 4m} U_{k-1}(t,m) \sum_{n}  h_w \lrp{ \frac{t^2-4m}{n^2} } \mu_{t,n,m}(N).
\end{equation}
Observe that $A_2(m,N,k)$ is a linear combination of the multiplicative functions $\mu_{t,n,m}$.
We can then use Lemma \ref{lemma:dirichlet-convolution} to bound the convolution $\beta * \mu_{t,n,m}$.

\begin{lemma} \label{lem:mu-multiplicative}
    Let $m \geq 1$ and $t,n$ be as in \eqref{eqn:A2m-formula}. Define $\mu^\nw_{t,n,m} := \beta * \mu_{t,n,m}$. Then for $N$ coprime to $m$,
    $$\lrabs{\mu^\nw_{t,n,m}(N)} \leq  2^{\omega(N)} \psi(n) \,2^{\omega(4m-t^2)} \sqrt{4m-t^2}.$$
    Here, $\omega(N)$ denotes the number of distinct prime divisors of $N$.
\end{lemma}
\begin{proof}
Recall from \eqref{eqn:A2m-formula} that 
    \begin{align}
    \mu_{t,n,m}(N) = \frac{\psi(N)}{\psi(N/\gcd(N,n))} \nu(N) \qquad \text{with} \qquad  \nu(N) = \sump_{\substack{c \mod N}} 1,
    \end{align}
    where the sum for $\nu(N)$ ranges over all  $c\in (\ZZ/N\ZZ)^\times$ that lift to a solution of the polynomial $x^2-tx+m\equiv0\mod N\cdot \gcd(N,n)$. 

    We now prove the desired bounds for $\mu^\nw_{t,n,m}$. Let $D := t^2-4m$ be the discriminant of $x^2 - tx + m$. We will show that for each prime $p$,
    \begin{align} \label{eqn:mu-new-local}
        \lrabs{\mu^\nw_{t,n,m}(p^r)} &\leq \begin{cases}
            2 & \text{if } p \nmid m,\ p \nmid D, \\
            4 p^{v_p(D)/2} \psi(p^{v_p(n)})  & \text{if } p \nmid m,\ p \mid D.
        \end{cases}
    \end{align}

    First, consider the case when $p \nmid m,\ p \nmid D$. Since $n^2 \mid D$, we also have $p \nmid n$, and so $\gcd(p^r,n) = 1$. This means that $\frac{\psi(p^r)}{\psi(p^r/\gcd(p^r,n))} = \frac{\psi(p^r)}{\psi(p^r)} = 1$. Additionally, observe that since $p \nmid m$, every  solution to $x^2 - tx + m \equiv 0 \mod p^r$ will necessarily be a unit modulo $p^r$.
    Thus, $\mu_{t,n,m}(p^r) = \nu(p^r)$ is precisely the number of solutions to the equation $x^2 - tx + m \equiv 0 \mod p^r$. And since $p \nmid D$, we have by Hensel's Lemma \cite[Chapter 2, Section 2, Proposition 2]{lang} that $\nu(p^r) = \nu(p)$ for all $r \geq 1$. Additionally, we have $\nu(p) \leq 2$ since $x^2 - tx + m $ is quadratic.
    
    Combining these observations, we obtain by Lemma \ref{lemma:dirichlet-convolution},
    \begin{align}
        \lrabs{\mu^\nw_{t,n,m}(p)} &= \lrabs{\mu_{t,n,m}(p) - 2} = \lrabs{\nu(p) - 2} \leq  2, \\
        \lrabs{\mu^\nw_{t,n,m}(p^2)} &= \lrabs{\mu_{t,n,m}(p^2) - 2 \mu_{t,n,m}(p) + 1} = \lrabs{\nu(p) - 2\nu(p) + 1} \leq  1, \\
        \lrabs{\mu^\nw_{t,n,m}(p^r)} &= \lrabs{\mu_{t,n,m}(p^r) - 2 \mu_{t,n,m}(p^{r-1}) + \mu_{t,n,m}(p^{r-2})} = \lrabs{\nu(p) - 2 \nu(p) + \nu(p)} = 0  \quad \text{for } r \geq 3.
    \end{align}
This verifies the first case of \eqref{eqn:mu-new-local}.

Next, consider the case of $p \nmid m,\  p \mid D$. Then 
    \begin{align}
        \frac{\psi(p^r)}{\psi(p^r/\gcd(p^r,n))}
        &= \begin{dcases}
            \frac{\psi(p^r)}{\psi(1)} &\text{if } r \leq v_p(n) \\
            \frac{\psi(p^r)}{\psi(p^{r-v_p(n)})} &\text{if } r > v_p(n)
        \end{dcases} \\
        &= \begin{dcases}
            \psi(p^r) \qquad\ \ \, &\text{if } r \leq v_p(n) \\
            p^{v_p(n)} \qquad\ \ \, &\text{if } r > v_p(n)
        \end{dcases} \\
        &\leq \psi(p^{v_p(n)}).
    \end{align}
    Also, observe that $\nu(p^r)$ will be bounded by the number of solutions to $x^2 - tx + m \equiv 0 \mod p^r$. We have from Huxley \cite[Page 194]{huxley} that the equation $x^2 - tx + m \equiv 0 \mod p^r$ has at most $2 p^{v_p(D)/2}$ solutions. Thus
    \begin{equation} \label{eqn:mu-bound-temp}
       \mu_{t,n,m}(p^r) =  \frac{\psi(p^r)}{\psi(p^r/\gcd(p^r,n))} \nu(N) \leq  \psi(p^{v_p(n)}) \cdot 2 p^{v_p(D)/2}. 
    \end{equation}
    This yields
    \begin{align}
        \lrabs{\mu^\nw_{t,n,m}(p^r)} &= \lrabs{\mu_{t,n,m}(p^r) - 2 \mu_{t,n,m}(p^{r-1}) + \mu_{t,n,m}(p^{r-2})} \\
        &\leq \max\lrp{\mu_{t,n,m}(p^r) + \mu_{t,n,m}(p^{r-2}),\ 2 \mu_{t,n,m}(p^{r-1})} \\
        &\leq 2 \cdot \psi(p^{v_p(n)}) \cdot 2 p^{v_p(D)/2},
    \end{align}
    where we interpret $\mu_{t,n,m}(p^{r-2})$ here as $0$ if $r = 1$. This verifies the second case of \eqref{eqn:mu-new-local}.

    Then from \eqref{eqn:mu-new-local}, for $N$ coprime to $m$, 
    $$\mu^\nw_{t,n,m}(N) \leq 2^{\omega(N)} 2^{\omega(D)} \psi(n) \sqrt{\lrabs{D}},$$
    as desired.
\end{proof}

We can then use this lemma to determine the asymptotic behavior of $A_{2}^\nw(m,N,k)$. Keep in mind that throughout this entire paper, all big-O notation is with respect to $N$ and $k$ ($m$ is just a fixed constant). 
\begin{corollary} \label{cor:A2-new-bound}
    Let $m \geq 1$ be fixed, and consider $N \geq 1$ coprime to $m$ and $k \geq 2$ even.
    Then 
    $$A_2^\nw(m, N, k) = O(m^{k/2}2^{\omega(N)}).$$
\end{corollary}

\begin{proof}
From \eqref{eqn:A2m-formula} and \eqref{eqn:Ainew-equals-beta-convolution-Ai},
\begin{align} \label{eqn:A2m-new-formula}
    A_{2}^\nw(m,N,k) &=  \frac{1}{2} \sum_{M \mid N} \beta\lrp{\frac{N}{M}} \sum_{t^2 < 4m} \sum_{n} U_{k-1}(t,m) h_w \lrp{ \frac{t^2-4m}{n^2} } \mu_{t,n,m}(M) \\
    &= \frac{1}{2} \sum_{t^2 < 4m} \sum_{n} U_{k-1}(t,m) h_w \lrp{ \frac{t^2-4m}{n^2} } \mu_{t,n,m}^\nw(N).
\end{align}
    Using the facts that $\lrabs{\rho} = \sqrt{m}$ and $\lrabs{\rho - \bar{\rho}} = \sqrt{4m-t^2}$, where $\rho, \bar{\rho}$ are the roots of $x^2 - tx + m = 0$, we have by the definition of $U_{k-1}$ from Lemma \ref{lemma:eichler-selberg-trace-formula} that
\begin{equation} \label{eqn:U-bounds}
    \lrabs{U_{k - 1}(t, m)} = \lrabs{ \frac{\rho^{k-1} - \bar{\rho}^{k-1}}{\rho - \bar{\rho}} } \leq \frac{ \lrabs{\rho^{k-1}} + \lrabs{\bar{\rho}^{k-1} } }{ \lrabs{ \rho - \bar{\rho} } } = \frac{2 m^{(k-1)/2}}{\sqrt{4m-t^2}} .
\end{equation} 
So by Lemma \ref{lem:mu-multiplicative},
\begin{equation}\label{eqn:U-mu-bound1}
    |U_{k-1}(t,m) \cdot \mu_{t,n,m}^\nw(N)| \leq 2 m^{(k-1)/2} \cdot 2^{\omega(N)} \cdot \psi(n) \cdot 2^{\omega(4m - t^2)}.
\end{equation}
Thus,
\begin{align}
    |A_2^\nw (m,N,k)| & \leq \frac{1}{2} \sum_{t^2 < 4m} \sum_n |U_{k-1}(t,m)| \cdot h_w \left( \frac{t^2 - 4m}{n^2} \right) |\mu_{t,n,m}^\nw(N)| \\
    & \leq \frac{1}{2} \sum_{t^2 < 4m} \sum_n h_w \left( \frac{t^2 - 4m}{n^2} \right) 2m^{(k-1)/2} \cdot 2^{\omega(N)} \cdot \psi(n) \cdot 2^{\omega(4m-t^2)} \\
    & = m^{(k-1)/2}2^{\omega(N)} \sum_{t^2 < 4m} 2^{\omega(4m - t^2)} \sum_n h_w \left( \frac{t^2 - 4m}{n^2} \right) \psi(n) 
    \label{eqn:A2new-explicit-bound}
    \\
    & = O\left(m^{k/2}2^{\omega(N)}\right),
\end{align}
as desired.
\end{proof}


\subsection{Bounding \texorpdfstring{$A_3^\nw(m,N,k)$}{A3new(m,N,k)}}

Next, recall that
\begin{equation}
    A_{3}(m,N,k) = \frac{1}{2} \sum_{d|m} \min(d,m/d)^{k-1} \sum_\tau \phi(\gcd(\tau, N/\tau)),
\end{equation}
where the inner summation runs over $\tau \mid N$  such that $\gcd(\tau, N/\tau) \mid (d-m/d)$.
Now denote this sum as
\begin{equation} \label{eqn:sigma-def}
    \Sigma_{m,d}(N) := \sum_{\tau} \phi(\gcd(\tau, N/\tau)).
\end{equation}
Observe that $A_3(m,N,k)$ is a linear combination of the $\Sigma_{m,d}$. We now show that these $\Sigma_{m,d}$ are multiplicative and use Lemma \ref{lemma:dirichlet-convolution} to bound the convolution $\beta * \Sigma_{m,d}$. 

\begin{lemma} \label{lem:sigma-multiplicative}
    Let $m,d \geq 1$ with $d \mid m$, $h := |d - \frac{m}{d}|$, and $\Sigma_{m,d}(N)$ be defined according to \eqref{eqn:sigma-def}. Then $\Sigma_{m,d}$ is multiplicative. Furthermore, define $\Sigma_{m,d}^\nw := \beta * \Sigma_{m,d}$. Then $\Sigma_{m,d}^\nw$ is bounded by
     \begin{equation}
         |\Sigma_{m,d}^\nw(N)| \leq 
         \begin{dcases}
            \frac{\sqrt{N}}{\pi_2(N)^2}  & \text{if } h = 0, \\
            h \cdot 4^{\omega(h)} & \text{if } h \neq 0.\\
        \end{dcases}
     \end{equation}
     Here,  
    $\pi_2(N)$ is the multiplicative function defined as
    \begin{align}
        \pi_2(N) := \prod_{p \mid N} \lrp{1+ \frac{1}{p-1}}.
    \end{align}
\end{lemma}

\begin{proof}
    Let $L$ and $M$ be coprime. Note that if $\tau \mid L$ and $\rho \mid M$, then $\gcd(\tau, L/\tau)$ and $\gcd(\rho, M/\rho)$ are coprime, and moreover $\gcd(\tau\rho, LM/\tau\rho) = \gcd(\tau, L/\tau)\cdot\gcd(\rho, M/\rho)$. Thus, 
    \begin{align}
        \Sigma_{m,d}(L) \Sigma_{m,d}(M) & = \sum_{\substack{\tau \mid L \\ (\tau, L/\tau) \mid h}} \sum_{\substack{\rho \mid M \\ (\rho, M/\rho) \mid h}} \phi(\gcd(\tau, L/\tau)) \phi(\gcd(\rho, M/\rho)) \\
        & = \sum_{\substack{\tau\mid L,\ \rho\mid M \\ (\tau, L/\tau)\mid h,\ (\rho, L/\rho)\mid h}} \phi(\gcd(\tau, L/\tau)\gcd(\rho, L/\rho)) \\
        & = \sum_{\substack{\tau\rho \mid LM \\ (\tau\rho, LM/\tau\rho)\mid h}} \phi(\gcd(\tau\rho, LM/\tau\rho)) \\
        &= \Sigma_{m,d}(LM).
    \end{align}
    This proves that $\Sigma_{m,d}$ is multiplicative. We can then define $\Sigma_{m,d}^\nw := \beta * \Sigma_{m,d}$. We divide the remaining proof into the case of $h=0$ and the case of $h\neq0$.
    
(1) First,  suppose $h = 0$.  
Then
    \begin{align} \label{eqn:sigma-h0}
        \Sigma_{m,d}(p^r) & = \sum_{\substack{\tau \mid p^r \\ (\tau,\,p^r/\tau) \mid 0 }} \phi(\gcd(\tau, p^r/\tau)) \\
        & = \sum_{0 \leq s \leq r} \phi(\gcd(p^s, p^{r-s})) \\
        & = \begin{dcases}
            2\cdot \sum_{0 \leq s \leq r/2} \phi(p^s) - \phi(p^{r/2}) & \text{if } r \text{ is even},\\
            2\cdot \sum_{0 \leq s \leq (r-1)/2} \phi(p^s) & \text{if } r \text{ is odd},\\
        \end{dcases} \\
        & = \begin{dcases}
            2p^{r/2} - \phi(p^{r/2}) & \text{if } r \text{ is even},\\
            2p^{(r-1)/2} & \text{if } r \text{ is odd}.\\
        \end{dcases} \label{eqn:sigma-m-d-local-formula-h0}
    \end{align}
    For the last step, we used the well-known formula $\sum_{d \mid N} \phi(d) = N$.
    
    We can now compute $\Sigma_{m,d}^\nw(p^r)$ explicitly. By 
    Lemma \ref{lemma:dirichlet-convolution}
    and \eqref{eqn:sigma-m-d-local-formula-h0},
    \begin{align}
        \Sigma_{m,d}^\nw(p) & = \Sigma_{m,d}(p) - 2 = 2 - 2 = 0,  \\
        \Sigma_{m,d}^\nw(p^2) & = \Sigma_{m,d}(p^2) - 2\Sigma_{m,d}(p) + 1 \\
        & = 2p - \phi(p) - 2 \cdot 2 + 1 \\
        & = p - 2.
    \end{align}
    For $r \geq 3$ odd,
    \begin{align}
        \Sigma_{m,d}^\nw(p^r) & = \Sigma_{m,d}(p^r) -2\Sigma_{m,d}(p^{r - 1}) + \Sigma_{m,d}(p^{r - 2}) \\
        & = 2p^{(r-1)/2} - 2\left(2p^{(r-1)/2} - \phi(p^{(r-1)/2})\right) + 2p^{(r-3)/2} \\
        & = 2p^{(r-1)/2} - 2\left(p^{(r-1)/2} + p^{(r-3)/2}\right) +  2p^{(r-3)/2} \\
        & = 0.
    \end{align}
    For $r \geq 3$ even,
    \begin{align}
        \Sigma_{m,d}^\nw(p^r) & = \Sigma_{m,d}(p^r) - 2\Sigma_{m,d}(p^{r-1}) + \Sigma_{m,d}(p^{r-2})\\
        & = 2p^{r/2} - p^{r/2 - 1}(p - 1) - 2\cdot 2p^{r/2 - 1} + 2p^{r/2 - 1} - p^{r/2 - 2}(p - 1) \\
        & = p^{r/2 - 2}\left(2p^2 - p^2 + p - 4p + 2p - p + 1\right) \\
        & = p^{r/2} \left( \frac{p^2 - 2p + 1}{p^2} \right) \\
        & = p^{r/2} \left(\frac{p - 1}{p}\right)^2.
    \end{align}
    To summarize, when $h = 0$,
    \begin{equation}
        \Sigma_{m,d}^\nw(p^r) = \begin{dcases}
            0 & \text{if } r \text{ is odd}, \\
            p - 2 & \text{if } r = 2, \\
            p^{r/2} \left(\frac{p-1}{p}\right)^2 & \text{if } r \geq 4 \text{ is even.} \\
        \end{dcases} \\
    \end{equation}
Observe that in each of these cases,
    \begin{align}
        \Sigma_{m,d}^\nw(p^r) & \leq p^{r/2}\lrp{\frac{p-1}{p}}^2 = \frac{p^{r/2}}{\lrp{1 + \frac{1}{p-1}}^2},
    \end{align}
    which yields
    \begin{align}
        \Sigma_{m,d}^\nw(N) & \leq \frac{\sqrt{N}}{\pi_2(N)^2},
    \end{align}
    as desired.
    
(2) Next, consider the case of $h \neq 0$. For $p \nmid h$ and $r \geq 1$,
    \begin{align} \label{eqn:sigma-p-coprime-h}
        \Sigma_{m,d}(p^r) & = \sum_{\tau} \phi(\gcd(\tau, p^r/\tau)) \\
        & = 2\cdot \phi(\gcd(1, p^r)) \\
        & = 2, \label{eqn:sigma-m-d-local-formula-hn0-1}
    \end{align}
    so by Lemma \ref{lemma:dirichlet-convolution},
    \begin{equation}
        \Sigma_{m,d}^\nw(p^r) = \begin{dcases}
            0 & \text{if } r = 1, \\
            -1 & \text{if } r = 2, \\
            0 & \text{if } r\geq 3. \\ 
        \end{dcases} \label{eqn:sigma-m-d-new-local-formula-hn0}
    \end{equation}
    For $p \mid h$, 
    \begin{align} \label{eqn:sigma-p-divides-h}
        \Sigma_{m,d}(p^r) & = \sum_{\tau} \phi(\gcd(\tau, p^r/\tau)) \\
        & = \sum_{\substack{0 \leq s \leq r \\ \gcd(p^s, p^{r-s}) \mid h}} \phi(\gcd(p^s, p^{r-s})) \\
        & \leq 2 \sum_{0 \leq s \leq v_p(h)} \phi(p^s) \\
        & = 2p^{v_p(h)}. \label{eqn:sigma-m-d-local-formula-hn0-2}
    \end{align}
    Then by Lemma \ref{lemma:dirichlet-convolution},
    \begin{align}
        |\Sigma_{m,d}^\nw(p^r)| & = |\Sigma_{m,d}(p^r) - 2\Sigma_{m,d}(p^{r - 1}) + \Sigma_{m,d}(p^{r - 2})|\\
        & \leq \max(\Sigma_{m,d}(p^r) + \Sigma_{m,d}(p^{r-2}), \ 2\Sigma_{m,d}(p^{r - 1}) )\\
        & \leq 2p^{v_p(h)} + 2p^{v_p(h)} \\
        & = 4p^{v_p(h)}, \label{eqn:sigma-m-d-new-local-formula-hn0-p-divides-h}
    \end{align}
    where we interpret $\Sigma_{m,d}(p^{r-2})$ here as 0 if $r = 1$. It immediately follows from \eqref{eqn:sigma-m-d-new-local-formula-hn0} and \eqref{eqn:sigma-m-d-new-local-formula-hn0-p-divides-h} that when $h \neq 0$ and $N$ is coprime to $m$,
    \begin{align} \label{eqn:sigma-new-hn0-global}
        |\Sigma_{m,d}^\nw(N)| & \leq  h \cdot 4^{\omega(h)},
    \end{align}
    as desired.
\end{proof}
We now use this lemma to determine the asymptotic behavior of $A_{3}^\nw(m,N,k)$.
\begin{corollary} \label{cor:A3-new-bound}
    Let $m \geq 1$ be fixed, and consider $N \geq 1$ coprime to $m$, and $k \geq 2$ even. Then 
    \begin{equation}
        A_3^\nw(m, N, k) = \begin{dcases}
            O\lrp{\frac{m^{k/2}\sqrt{N}}{\pi_2(N)^2}} & \text{if } m \text{ is a perfect square,} \\
            O\lrp{m^{k/2}} & \text{if }  m \text{ is not a perfect square.} \\
        \end{dcases}
    \end{equation}
\end{corollary}

\begin{proof}
    Since $\min(d, m/d)^{k-1} \leq m^{(k-1)/2}$,
    \begin{align}
        |A_{3}^\nw(m,N,k)| &= \lrabs{\frac{1}{2} \sum_{M \mid N} \beta\lrp{\frac{N}{M}} \sum_{d|m} \min(d,m/d)^{k-1} \Sigma_{m,d}(M)} \\
        & = \lrabs{\frac{1}{2} \sum_{d|m} \min(d, m/d)^{k-1} \Sigma_{m,d}^\nw(N)} \\
        & \leq \frac{1}{2} m^{(k-1)/2} \sum_{d|m} \lrabs{\Sigma_{m,d}^\nw(N)}.
    \end{align}
    The desired result then follows immediately from Lemma \ref{lem:sigma-multiplicative} (since $h = 0$ only for $d = \sqrt{m}$, which requires $m$ to be a perfect square).
\end{proof}


\subsection{Bounding \texorpdfstring{$A_4^\nw(m,N,k)$}{A4new(m,N,k)}}

For $t \geq 0$, we use the notation $\sigma_t(m) = \sum_{d \mid m} d^t$. Then since $N$ is coprime to $m$,
\begin{equation} \label{eqn:A4-Section-3}
    A_{4}(m,N,k) = 
    \begin{dcases} 
      \sigma_1(m) & \text{if $k=2$} \\
      0 & \text{if $k>2$} .\\
   \end{dcases} \\
\end{equation}

Observe that $A_4(m,N,k)$ is a multiple of the constant multiplicative function $\mathbf{1}(N) = 1$. The following Lemma then follows immediately from Lemma \ref{lemma:dirichlet-convolution}. 

\begin{lemma} \label{lem:1-new-bound}
    Define the multiplicative function $\mathbf{1}^\nw := \beta * \mathbf{1}$. Then $\lrabs{\mathbf{1}^\nw(N)} \leq 1$.
\end{lemma}

This lemma then yields
\begin{equation} \label{eqn:A4-big-O}
    |A_4^\nw(m,N,k)| \leq \sigma_1(m) = O(1).
\end{equation}
(Remember that we are using big-O notation with respect to $N$ and $k$.)


\section{Proof of Theorem \ref{thm:main-theorem} for trivial character} \label{sec:main-thm}

In this section, we prove Theorem \ref{thm:main-theorem} for trivial character. First, we define the functions $\theta_i(N)$ which will be used to express the error terms for certain trace estimates.

\begin{lemma} \label{lem:theta-bounds}
    Define
    \begin{equation}
    \begin{alignedat}{3}
        \theta_1(N) &= \frac{ \sqrt{N}}{\psi^\nw(N)\pi_2(N)^2},
        \qquad\qquad
        & \theta_2(N) &= \frac{4^{\omega(N)}}{\psi^\nw(N)}, \\
        \theta_3(N) &= \frac{2^{\omega(N)}}{\psi^\nw(N)}, 
        \qquad\qquad 
        & \theta_4(N) &= \frac{1}{\psi^\nw(N)}.
    \end{alignedat}
    \end{equation}
    Then each $\theta_i(N) \longrightarrow 0$ as $N \longrightarrow \infty$.
\end{lemma}

\begin{proof}
    Recall that $\pi_1(N) = \prod_{p \mid N} \lrp{1 + \frac{p+1}{p^2-p-1}}$ and $\pi_2(N) = \prod_{p \mid N} \lrp{1+ \frac{1}{p-1}}$. Now, observe that $\frac13 \pi_1(N),\pi_2(N) \leq 2^{\omega(N)} = O(N^{\varepsilon})$ for any $\varepsilon > 0$ \cite[Sections~18.1,~22.13]{hardy-wright}.
    Thus since $\psi^\nw(N) \geq \frac{N}{\pi_1(N)}$ by Lemma \ref{lem:psi-new}, we have that $\theta_1(N) = O(N^{-1/2+\varepsilon}) \longrightarrow 0$ as $N \longrightarrow \infty$, and for $i \in \{2, 3, 4\}$, $\theta_i(N) = O(N^{-1 + \varepsilon}) \longrightarrow 0$ as $N \longrightarrow \infty$.
\end{proof}

Our proof of Theorem \ref{thm:main-theorem} for trivial character will be divided into two cases: when $m$ is not a perfect square, and when $m$ is a perfect square. First, we present two lemmas which estimate $\Tr T_m^\nw(N,k)$. 
\begin{lemma} \label{lem:m-not-square}
    Let $m$ fixed not be a perfect square, and consider $N \geq 1$ coprime to $m$ and $k \geq 2$ even. Then
    $$\Tr T_m^\nw(N,k) = O(2^{\omega(N)} m^{k/2}).$$
\end{lemma}

\begin{proof}
    We consider each of the $A_i^\nw (m,N,k)$ terms from \eqref{eqn:new-eichler-selberg-trace-formula} separately. First, since $m$ is not a perfect square, $A_1^\nw (m,N,k) = 0$.
    Next, by Corollary \ref{cor:A2-new-bound}, $$A_2^\nw(m, N, k) = O\left(m^{k/2}2^{\omega(N)} \right).$$
    By Corollary \ref{cor:A3-new-bound}, since $m$ is not a perfect square,
    $$A_3^\nw(m, N, k) = O\lrp{m^{k/2}}.$$
    And from \eqref{eqn:A4-big-O}, $A_4^\nw(m, N, k) = O(1)$. Thus, 
    \begin{align}
        \Tr T_m^\nw(N,k) & = A_1^\nw(m, N, k) - A_2^\nw(m, N, k) - A_3^\nw(m, N, k) + A_4^\nw(m, N, k) \\
        & =  O \lrp{m^{k/2}2^{\omega(N)}} + O\lrp{m^{k/2}} + O(1) \\
        & = O \lrp{m^{k/2}2^{\omega(N)}},
    \end{align}
    completing the proof.
\end{proof}

\begin{lemma} \label{lem:m-perfect-square}
    Let $m$ fixed be a perfect square, and consider $N \geq 1$ coprime to $m$ and $k \geq 2$ even. Then
    $$\Tr T_m^\nw(N,k) = \frac{k-1}{12} \psi^\nw(N) m^{k/2-1}    + O\lrp{\frac{m^{k/2} \sqrt{N}}{\pi_2(N)^2}}.$$
\end{lemma}
\begin{proof}
    First, by \eqref{eqn:A1m-new-formula}, 
    $$A_1^\nw(m, N, k) = \chi_0(\sqrt{m})\frac{k-1}{12} \psi^\nw(N) m^{k/2-1} = \frac{k-1}{12} \psi^\nw(N) m^{k/2-1}.$$
    Next, as in Lemma \ref{lem:m-not-square}, we have the bounds $A_2^\nw(m, N, k) = O\lrp{2^{\omega(N)}m^{k/2}}$ and \\ $A_4^\nw(m, N, k) = O(1)$. Additionally, by Corollary \ref{cor:A3-new-bound}, 
    $$A_3^\nw(m,N,k)=O\lrp{\frac{m^{k/2}\sqrt{N}}{\pi_2(N)^2}}$$
    since $m$ is a perfect square.
    
    Thus by the trace formula,
    \begin{align}
        \Tr T_m^\nw(N, k) & = A_1^\nw(m, N, k) - A_2^\nw(m,N,k) - A_3^\nw(m, N, k) + A_4^\nw(m, N, k) \\
        & = \frac{k-1}{12} \psi^\nw(N) m^{k/2-1} + O\left(2^{\omega(N)}m^{k/2}\right) + O\lrp{\frac{m^{k/2}\sqrt{N}}{\pi_2(N)^2}} + O(1) \\
        & = \frac{k-1}{12} \psi^\nw(N) m^{k/2-1} + O\lrp{\frac{m^{k/2}\sqrt{N}}{\pi_2(N)^2}},
    \end{align}
    completing the proof.  
\end{proof}

We now prove Theorem \ref{thm:main-theorem} for trivial character in two separate cases. Proposition \ref{prop:nsquare} addresses the case when $m$ is not a perfect square, and Proposition \ref{prop:square} addresses the case when $m$ is a perfect square. 
\begin{proposition} \label{prop:nsquare}
    Let $m$ be fixed and not a perfect square, and consider $N \geq 1$ coprime to $m$ and $k \geq 2$ even. Then $a_2^\nw(m,N,k) < 0$ for all but finitely many pairs $(N,k)$.
\end{proposition}

\begin{proof}
    Recall from Lemma \ref{lem:a2-coeff-formula} that
\begin{align} \label{eqn:a2}
    a_2^\nw(m,N,k) = \frac{1}{2} \lrb{ \left(\Tr T_m^\nw\right)^2 - \sum_{d \mid m} d^{k-1} \Tr T_{m^2/d^2}^\nw }.
\end{align}
  Since $m^2/d^2$ is a perfect square, we can employ Lemma \ref{lem:m-perfect-square} on the $T^\nw_{m^2/d^2}$ terms in \eqref{eqn:a2} and obtain
\begin{align}
    \sum_{d \mid m} d^{k-1} \Tr T_{m^2/d^2}^\nw & = \sum_{d \mid m} d^{k-1} \lrb{ \frac{k-1}{12} \psi^\nw(N) \lrp{\frac{m^2}{d^2}}^{k/2-1} + O\lrp{\lrp{\frac{m^2}{d^2}}^{k/2}\frac{\sqrt{N}}{\pi_2(N)^2}} } \\
    & = \psi^\nw(N) m^{k - 2} \sum_{d \mid m} \left(\frac{k-1}{12} d + O\lrp{\theta_1(N)}\right) \\ 
    & = \psi^\nw(N) m^{k - 2} \sigma_1(m) \left(\frac{k - 1}{12} + O(\theta_1(N)) \right). \label{eqn:sum-of-square-traces}
\end{align}
Since $m$ is not a square, we can also use Lemma \ref{lem:m-not-square} on the $(\Tr T_m^\nw)^2$ term in \eqref{eqn:a2} to obtain
\begin{align}
    a_2^\nw(m,N,k) & = \frac{1}{2} \lrb{ O\lrp{2^{\omega(N)}m^{k/2}}^2 - \psi^\nw(N) m^{k - 2} \sigma_1(m) \left( \frac{k-1}{12} + O(\theta_1(N)) \right)} \\
    & = \frac{1}{2}\psi^\nw(N)m^{k - 2}\sigma_1(m) \lrb{ -\frac{k-1}{12} + O\lrp{\theta_1(N)} + O\lrp{\theta_2(N)}} \label{eqn:a2new-final-formula-temp}
\end{align}
Now, note that $\frac{1}{2}\psi^\nw(N)m^{k - 2}\sigma_1(m) > 0$ (by Lemma \ref{lem:psi-new}) and $\frac{k-1}{12} \geq \frac{1}{12}$ for all $k \geq 2$. Thus since the $\theta_i(N) \longrightarrow 0$,  the $O(\theta_1(N)) + O(\theta_2(N))$ term from \eqref{eqn:a2new-final-formula-temp} will be $< \frac{1}{12}$ in magnitude and hence $a_2^\nw(m, N, k) < 0$ for sufficiently large $N$. Then for each of the finitely many remaining fixed values of $N$, we also have from \eqref{eqn:a2new-final-formula-temp} that $a_2^\nw(m,N,k) < 0$ for sufficiently large $k$. Thus
$a_2^\nw(m,N,k) < 0$ for all but finitely many pairs $(N, k)$.
\end{proof}

\begin{proposition} \label{prop:square}
    Let $m$ fixed be a perfect square, and consider $N \geq 1$ coprime to $m$ and $k \geq 2$ even. Then $a_2(m,N,k) > 0$ for all but finitely many pairs $(N,k)$.
\end{proposition}

\begin{proof}
    By \eqref{eqn:sum-of-square-traces} and Lemma \ref{lem:m-perfect-square},
\begin{align}
    a_2^\nw(m,N,k) & = \frac{1}{2} \Bigg[ \left(\Tr T_m^\nw\right)^2 - \sum_{d \mid m} d^{k-1} \Tr T_{m^2/d^2}^\nw \Bigg]\\
    & = \frac{1}{2} \Bigg[ \left(\frac{k-1}{12} \psi^\nw(N) m^{k/2-1} + O\left(\frac{m^{k/2}\sqrt{N}}{\pi_2(N)^2}\right) \right)^2  \\
    & \qquad \qquad  + \psi^\nw(N) m^{k - 2} \sigma_1(m) \left( -\frac{k-1}{12} + O(\theta_1(N)) \right) \Bigg] \\
    & = \frac{(k - 1)^2}{288}\psi^\nw(N)^2m^{k - 2} + O\lrp{\frac{k-1}{12}\psi^\nw(N) \cdot \frac{m^{k-1}\sqrt{N}}{\pi_2(N)^2}} + O \lrp{\frac{m^k N}{\pi_2(N)^4}}\\
    & \qquad \qquad - \frac{k - 1}{12} \psi^\nw(N) O(m^k) + \psi^\nw(N) O(m^k\theta_1(N)) \\
    & = (k-1)\psi^\nw(N)^2m^{k-2} \Bigg[\frac{k-1}{288} + O(\theta_1(N)) + O\lrp{\theta_1(N)^2} \\
    & \mspace{250mu} + O\lrp{\theta_4(N)} + O\lrp{\theta_1(N)\theta_4(N)} \Bigg] \label{eqn:a2new-final-formula-temp2}
\end{align}
    Now, note that $(k-1)\psi^\nw(N)^2m^{k-2}  > 0$ (by Lemma \ref{lem:psi-new}) and $\frac{(k-1)}{288} \geq \frac{1}{288}$ for all $k \geq 2$. Thus since the $\theta_i(N) \longrightarrow 0$, we have by \eqref{eqn:a2new-final-formula-temp2} that for sufficiently large $N$, $a_2^\nw(m, N, k) > 0$. Then for each of the finitely many remaining values of $N$, we also have by \eqref{eqn:a2new-final-formula-temp2} that $a_2^\nw(m, N, k) > 0$ for sufficiently large $k$. Thus $a_2^\nw(m, N, k) > 0$ for all but finitely many pairs $(N, k)$.
\end{proof}
Propositions \ref{prop:nsquare} and \ref{prop:square} combine to imply Theorem \ref{thm:main-theorem} for trivial character.

\section{Computing \texorpdfstring{$a_2^{\nw}(2,N,k)$}{a2\^new(2,N,k)} and \texorpdfstring{$a_2^{\nw} (4,N,k)$}{a2\^new(4,N,k)} } \label{sec:T2}

To illustrate the method given in Section \ref{sec:main-thm}, we now compute the specific pairs $(N,k)$ for which $a_2^\nw(2,N,k)$ and $a_2^\nw(4,N,k)$ vanish, verifying Theorem \ref{thm:m=2-4-Theorem}.

\subsection{The Nonvanishing of \texorpdfstring{$a_2^{\nw}(2,N,k)$}{a2\^new(2,N,k)}}

From Lemma \ref{lem:a2-coeff-formula},
\begin{equation} \label{eqn:a2-new}
    a_2(2,N,k)^\nw = \frac{1}{2} \lrb{ (\Tr T_2^\nw)^2 - \Tr T_4^\nw - 2^{k-1} \Tr T_1^\nw }.
\end{equation}

We first bound the $\Tr T_2^\nw$ term of \eqref{eqn:a2-new} explicitly.

\begin{lemma} \label{lem:T2-square-bound}
    We have the following bound:
    \begin{equation}
        \frac{(\Tr T_2^\nw)^2}{\psi^\nw(N) 2^k} \leq 32\,\theta_2(N) + 16\sqrt{2}\,\theta_3(N) + 4\,\theta_4(N),
    \end{equation}
    where the $\theta_i(N)$ are as defined in Lemma \ref{lem:theta-bounds}.
\end{lemma}

\begin{proof}
    Since $2$ is not a perfect square, $A_1^\nw(2,N,k)=0$ by \eqref{eqn:A1m-new-formula}.
    
    Then, by \eqref{eqn:A2new-explicit-bound}  
    and the values of $h_w$ from Table \ref{table:weighted-class-numbers}, 
    \begin{align}
        |A_2^\nw(2, N, k)| & \leq 2^{(k-1)/2}2^{\omega(N)} \cdot \sum_{t^2 < 8} 2^{\omega(8 - t^2)} \sum_n h_w \left( \frac{t^2 - 8}{n^2} \right) \psi(n) \\
        & \leq 2^{(k-1)/2} 2^{\omega(N)} \cdot  \left[ 2^{\omega(8)} \cdot 1 \cdot \psi(1) + 2 \cdot 2^{\omega(7)} \cdot 1 \cdot \psi(1) + 2 \cdot 2^{\omega(4)} \cdot \frac{1}{2} \cdot \psi(1) \right]\\
        & = 8 \cdot 2^{(k-1)/2} \cdot 2^{\omega(N)}.
    \end{align}

    Then by \eqref{eqn:A3m-formula} and Lemma \ref{lem:sigma-multiplicative}, and using the fact that the $d=d_0$ and $d=2/d_0$ terms in the sum coincide,
    \begin{align}
        \lrabs{A_3^\nw(2,N,k)} &= \lrabs{\frac{1}{2} \sum_{d \mid 2} \min (d,2/d)^{k-1} \cdot  \Sigma_{2,d}^\nw(N)} = \lrabs{\Sigma_{2,1}^\nw(N)} \leq 1 \cdot 4^{\omega(1)} = 1.
    \end{align}

    Finally, by \eqref{eqn:A4-big-O}, $|A_4^\nw(2,N,k)| \leq \sigma_1(2) = 3$. Thus 
    \begin{align}
        \frac{(\Tr T_2^\nw)^2}{\psi^\nw(N)2^k} &\leq \frac{1}{\psi^\nw(N)2^k}\lrp{|A_2^\nw(2,N,k)| + |A_3^\nw(2,N,k)| + |A_4^\nw(2,N,k)|}^2 \\ &\leq \frac{1}{\psi^\nw(N)2^k} \lrp{8 \cdot 2^{(k-1)/2} \cdot 2^{\omega(N)} + 4}^2 \\
        & = 64 \frac{\theta_2(N)}{2} + 64 \frac{\theta_3(N)}{2^{(k+1)/2}} + 16\frac{\theta_4(N)}{2^k} \\
        & \leq 32\,\theta_2(N) + 16\sqrt{2} \, \theta_3(N) + 4\,\theta_4(N),
    \end{align}
    as desired, since $k \geq 2$.
\end{proof}

We now bound the error terms of the $\Tr T_4^\nw$ term of \eqref{eqn:a2-new}.

\begin{lemma} \label{lem:T4-A14-bound}
    We have the following bound:
    \begin{equation}
        \lrabs{\frac{\Tr T_4^\nw - A_1^\nw(4,N,k)}{\psi^\nw(N) 2^k}} \leq \frac{1}{4} \theta_1(N) + \frac{41}{2} \theta_3(N) + \frac{19}{4}\theta_4(N).
    \end{equation}
\end{lemma}

\begin{proof}

    Using the bound of \eqref{eqn:A2new-explicit-bound} and the values of $h_w$ given in Table \ref{table:weighted-class-numbers},
    \begin{align}
        A_2^\nw(4,N,k) & \leq 4^{(k-1)/2}2^{\omega(N)} \sum_{t^2 < 16} 2^{\omega(16 - t^2)} \sum_n h_w \left( \frac{t^2 - 16}{n^2} \right) \psi(n)\\
        & \leq 4^{(k-1)/2} \cdot 2^{\omega(N)} \left[2^{\omega(16)}\left(1 \cdot \psi(1) + \frac{1}{2} \cdot \psi(2)\right) + 2 \cdot 2^{\omega(15)} \cdot 2 \cdot \psi(1) \right. \\
        & \qquad \left. + 2 \cdot 2^{\omega(12)} \left(1 \cdot \psi(1) + \frac{1}{3} \cdot \psi(2)\right) + 2 \cdot 2^{\omega(7)} \cdot 1 \cdot \psi(1) \right] \\
        & = \frac{41}{2} \cdot 2^k \cdot 2^{\omega(N)}.
    \end{align}

    Then by \eqref{eqn:A3m-formula} and Lemma \ref{lem:sigma-multiplicative},
    \begin{align}
        \lrabs{A_3^\nw(4,N,k)} &= \lrabs{\frac{1}{2} \sum_{d \mid 4} \min (d,4/d)^{k-1} \cdot \Sigma_{4,d}^\nw(N)} \\
        &= \lrabs{\Sigma_{4,1}^\nw(N) + \frac12 2^{k-1} \Sigma_{4,2}^\nw(N)} \\
        &\leq 3 \cdot 4^{\omega(3)} + \frac12 2^{k-1} \frac{\sqrt{N}}{\pi_2(N)^2} \\
        & = 12 + 2^{k-2} \frac{\sqrt{N}}{\pi_2(N)^2}.
    \end{align}
    
    Finally, by \eqref{eqn:A4-big-O} we have $|A_4^\nw(4,N,k)| \leq \sigma_1(4) = 7.$ Now, since $k \geq 2$, we see that
    \begin{align}
        \lrabs{\frac{\Tr T_4^\nw - A_1^\nw(4,N,k)}{\psi^\nw(N) 2^k}} & \leq \frac{|A_2^\nw(4, N, k)| + |A_3^\nw(4, N, k)| + |A_4^\nw(4, N, k)|}{\psi^\nw(N) 2^k} \\
        & \leq \frac{1}{\psi^\nw(N)} \left( \frac{41}{2} \cdot 2^{\omega(N)} + 3 + \frac{1}{4} \frac{\sqrt{N}}{\pi_2(N)^2} + \frac{7}{4} \right) \\
        & = \frac{41}{2}  \theta_3(N) + \frac{1}{4} \theta_1(N) + \frac{19}{4}\theta_4(N), \\
    \end{align}
    as desired.
\end{proof}

We now bound the error terms of the $\Tr T_1^\nw$ term of \eqref{eqn:a2-new}.

\begin{lemma} \label{lem:T1-A11-bound} 
     We have the following bound:
    \begin{equation}
        \lrabs{\frac{\Tr T_1^\nw - A_1^\nw(1,N,k)}{\psi^\nw(N)}} \leq \frac{1}{2}  \theta_1(N) + \frac{7}{3}  \theta_3(N) + \theta_4(N).
    \end{equation}
\end{lemma}

\begin{proof}
    Again using the bound from \eqref{eqn:A2new-explicit-bound} and the values of $h_w$ from Table \ref{table:weighted-class-numbers}, 
    \begin{align}
         \lrabs{A^\nw_2(1, N, k)} 
         &\leq 2^{\omega(N)} \sum_{t^2 < 4} 2^{\omega(4 - t^2)} \sum_n h_w \left( \frac{t^2 - 4}{n^2} \right) \psi(n) \\ 
         & =  2^{\omega(N)} \lrb{2^{\omega(4)} \cdot \frac{1}{2} \cdot  \psi(1)  + 2 \cdot 2^{\omega(3)} \cdot \frac{1}{3} \cdot \psi(1)}  \\
         & = \frac{7}{3} \cdot 2^{\omega(N)}.
    \end{align}
    Then, by \eqref{eqn:A3m-formula} and Lemma \ref{lem:sigma-multiplicative},
    \begin{align}
        \lrabs{A^\nw_{3,1}} &= \frac{1}{2}\lrabs{\sum_{d\mid1}\min\left(d,1/d \right)^{k-1}\Sigma^\nw_{1,d}} = \frac{1}{2}\lrabs{\Sigma^\nw_{1,1}} \leq \frac{1}{2}\frac{\sqrt{N}}{\pi_2(N)^2}.
    \end{align}
    Finally, $\lrabs{A^\nw_4(1, N, k)} \leq \sigma_1(1) = 1$ by \eqref{eqn:A4-big-O}. Thus by \eqref{eqn:new-eichler-selberg-trace-formula}, 
    \begin{align}
        \lrabs{\frac{\Tr T_1^\nw - A_1^\nw(1,N,k)}{\psi^\nw(N)}} &= \lrabs{\frac{-A^\nw_2-A^\nw_3+A^\nw_4}{\psi^\nw(N)}} \\
        & \leq \frac{\lrabs{A^\nw_2} + \lrabs{A^\nw_3}+ \lrabs{A^\nw_4}}{\psi^\nw(N)} \\
        &\leq\frac{\frac{7}{3} 2^{\omega(N)} + \frac{1}{2}\frac{\sqrt{N}}{\pi_2(N)^2} + 1}{\psi^\nw(N)} \\
        & \leq \frac{7}{3}\theta_3(N) + \frac{1}{2}\theta_1(N) + \theta_4(N), \\
    \end{align}
    as desired.    
\end{proof}

Before proving the first part of Theorem \ref{thm:m=2-4-Theorem}, we give explicit bounds for each of the $\theta_i(N)$ defined in Lemma \ref{lem:theta-bounds}. 

\begin{lemma} \label{lem:computable-theta-bounds}
We have the following bounds for the $\theta_i(N)$ defined in Lemma \ref{lem:theta-bounds}:
    \begin{equation}
    \begin{alignedat}{4}
        \theta_1(N) &\leq \frac{1}{\sqrt{N}}, 
        \qquad\qquad & \theta_2(N) &\leq \frac{1304.3}{N^{37/64}}, \\
        \theta_3(N) &\leq \frac{125.28}{N^{25/32}}, 
        & \theta_4(N) &\leq \frac{12.033}{N^{63/64}}.
    \end{alignedat}
    \end{equation}
\end{lemma}

\begin{proof}
    First we will prove explicit bounds on the multiplicative functions $\pi_1(N)$ and $2^{\omega(N)}$. 
    In particular, 
    following the method of \cite[Lemma 2.4]{ross}, 
    we show that
    \begin{align}
        \pi_1(N) \leq 12.033 \cdot N^{1/64} 
        \qquad  \text{and} \qquad
        2^{\omega(N)} \leq 10.411 \cdot N^{13/64}. 
    \end{align}
    Recall from Lemma \ref{lem:psi-new} that 
        $\pi_1(N) = \prod_{p \mid N} \lrp{1 + \frac{p+1}{p^2-p-1}}$.
     One can verify that $\lrp{1 + \frac{p+1}{p^2-p-1}} < p^{1/64}$ for all primes $p \geq 23$. Now let $c_p = 1$ for primes $p \geq 23$, and $c_p = \lrp{1 + \frac{p+1}{p^2-p-1}} / p^{1/64}$ for $2 \leq p \leq 19$.
    Then 
    \begin{align}
        \pi_1(N) &= \prod_{p \mid N} \lrp{1 + \frac{p+1}{p^2-p-1}} \\
        &\leq \prod_{p \mid N} c_p \cdot p^{1/64} \\
        &\leq \prod_{p^r \parallel N} c_p \cdot p^{r/64} \\
        &\leq c_2 \cdots c_{19} \cdot N^{1/64} \\
        &\leq 12.033 \cdot N^{1/64}.
    \end{align}
    
    Similarly, one can verify that $2 < p^{13/64}$ for all primes $p \geq 31$. Now let $c_p' = 1$ for primes $p \geq 31$, and $c_p' = 2 / p^{13/64}$ for $2 \leq p \leq 29$.
    Then 
    \begin{align}
        2^{\omega(N)} &= \prod_{p \mid N} 2 \\
        &\leq c'_2 \cdots c'_{29} \cdot N^{13/64} \\
        &\leq 10.411 \cdot N^{13/64}.
    \end{align}
    
    Now, recall from Lemma \ref{lem:sigma-multiplicative} that $\pi_2(N) = \prod_{p \mid N} \lrp{1+ \frac{1}{p-1}}$. Observe that
    \begin{align}
        \frac{\pi_1(N)}{\pi_2(N)^2} & = \prod_{p \mid N} \frac{p^2}{p^2 - p - 1} \cdot \lrp{\frac{p-1}{p}}^2  = \prod_{p \mid N} \frac{p^2 - 2p + 1}{p^2 - p - 1} \leq 1.
    \end{align}
    Thus by Lemma \ref{lem:psi-new},
    \begin{align}
        \theta_1(N) & = \frac{\sqrt{N}}{\psi^\nw(N)\pi_2(N)^2}  \leq \sqrt{N} \cdot \frac{\pi_1(N)}{N} \cdot \frac{1}{\pi_2(N)^2}  \leq \frac{1}{\sqrt{N}}.
    \end{align}
Next, 
    \begin{align}
        \theta_4(N) & = \frac{1}{\psi^\nw(N)}  \leq \frac{\pi_1(N)}{N}  \leq \frac{12.033N^{1/64}}{N}  = \frac{12.033}{N^{63/64}}.
    \end{align}
This now allows us to bound $\theta_2(N)$;
    \begin{align}
        \theta_2(N) & = \frac{4^{\omega(N)}}{\psi^\nw(N)}  \leq \lrp{10.411 \cdot N^{13/64}}^2 \cdot \frac{12.033}{N^{63/64}}  \leq \frac{1304.3}{N^{37/64}}.
    \end{align}
And finally we can bound $\theta_3(N)$;
    \begin{align}
        \theta_3(N) & = \frac{2^{\omega(N)}}{\psi^\nw(N)}  \leq 10.411 \cdot N^{13/64} \cdot \frac{12.033}{N^{63/64}}  \leq \frac{125.28}{N^{25/32}},
    \end{align}
completing the proof.
\end{proof}

We now prove the first part of Theorem \ref{thm:m=2-4-Theorem}, computing the complete list of pairs $(N,k)$ for which $a^\nw_2(2,N,k)$ vanishes.

\begin{proof}
    By \eqref{eqn:a2-new} and \eqref{eqn:A1m-new-formula}, we have
\begin{align}
    a_2^\nw(2,N,k) &= \frac{1}{2} \lrb{ (\Tr T_2^\nw)^2 - \Tr T_4^\nw - 2^{k-1} \Tr T_1^\nw } \\
    &= \frac{1}{2} \Big{[} (\Tr T_2^\nw)^2 - A_1^\nw(4,N,k) - \left(\Tr T_4^\nw - A_1^\nw(4,N,k)\right) \\
    & \qquad - 2^{k-1} A_1^\nw(1,N,k) -2^{k-1}\left(\Tr T_1^\nw - A_1^\nw (1,N,k)\right) \Big{]} \\
    &= \psi^\nw(N) 2^{k-1} \left[ -\frac{A_1^\nw(4,N,k)}{\psi^\nw(N) 2^k} - \frac{1}{2} \frac{A_1^\nw(1,N,k)}{\psi^\nw(N)} \right. \\
    & \qquad \left. + \frac{(\Tr T_2^\nw)^2}{\psi^\nw(N) 2^k} - \frac{\Tr T_4^\nw - A_1^\nw(4,N,k)}{\psi^\nw(N) 2^k} - \frac{1}{2} \frac{\Tr T_1^\nw - A_1^\nw(1,N,k)}{\psi^\nw(N)} \right] \\
    & = \psi^\nw(N) 2^{k-1} \left[-\frac{k-1}{12} \cdot \frac{4^{k/2-1}}{2^k} - \frac{k-1}{24} \cdot 1^{k/2-1}\right. \\
    & \qquad \left. + \frac{(\Tr T_2^\nw)^2}{\psi^\nw(N) 2^k} - \frac{\Tr T_4^\nw - A_1^\nw(4,N,k)}{\psi^\nw(N) 2^k} - \frac{1}{2} \frac{\Tr T_1^\nw - A_1^\nw(1,N,k)}{\psi^\nw(N)} \right] \\
    & = \psi^\nw(N) 2^{k-1} \lrb{ -\frac{k-1}{16} + E(N,k)},
\end{align}
where $E(N,k)$ denotes the three error terms. By Lemmas \ref{lem:T2-square-bound}, \ref{lem:T4-A14-bound}, and \ref{lem:T1-A11-bound},
\begin{align}
    |E(N,k)| &= \left| \frac{(\Tr T_2^\nw)^2}{\psi^\nw(N) 2^k} - \frac{\Tr T_4^\nw - A_1(4,N,k)^\nw}{\psi^\nw(N) 2^k} - \frac{1}{2} \frac{\Tr T_1^\nw - A_1(1,N,k)^\nw}{\psi^\nw(N)} \right| \\
    & \leq 32\,\theta_2(N) + 16\sqrt{2}\,\theta_3(N) + 4\,\theta_4(N) + \frac{1}{4}\theta_1(N) + \frac{41}{2}\theta_3(N) + \frac{19}{4}\theta_4(N) \\
    & \qquad + \frac{1}{2}\lrp{\frac{1}{2}\theta_1(N) + \frac{7}{3}\theta_3(N) + \theta_4(N)} \\
    & = \frac{1}{2}\theta_1(N) + 32\, \theta_2(N) + \lrp{ 16\sqrt{2}+\frac{65}{3} }\theta_3(N) + \frac{37}{4}\theta_4(N).
\end{align}

Then by the explicit $\theta_i(N)$ bounds given in Lemma \ref{lem:computable-theta-bounds},
\begin{align}
    |E(N, k)| \leq \frac{1}{2\sqrt{N}} + \frac{32 \cdot 1304.3}{N^{37/64}} + \lrp{ 16\sqrt{2}+\frac{65}{3} } \cdot \frac{125.28}{N^{25/32}} + \frac{37 \cdot 12.033}{4N^{63/64}},
\end{align}
which is clearly monotonically decreasing. We then observe that when $N = 1.19130 \cdot 10^{10}$, $|E(N, k)| \leq 0.0624997 < \frac{1}{16}$. Thus, for all $N \geq 1.19130 \cdot 10^{10}$ and $k \geq 2$ even, $a_2^\nw(2, N, k) < 0$. We then compute all $N < 1.19130 \cdot 10^{10}$ in Sage \cite{sagemath}, obtaining that $a_2^\nw(2,N,k) = 0$ for thirty-eight different pairs $(N,k)$  \cite[Table A]{ross-code}. Comparing with \cite[Tables 6.2, 6.3]{ross}, thirty-six of these pairs have $\dim S_k(\Gamma_0(N)) < 2$. The two remaining pairs for which $a_2^\nw(2,N,k)$
nontrivially vanishes are $(37,2)$ and  $(57,2)$, proving the desired result.
\end{proof}

We also note from \cite[Table A]{ross-code} that $a_2^\nw(2,N,k) > 0$ for exactly five pairs $(N,k)$: $a_2^\nw(2,3,16)=16848$, $a_2^\nw(2,3,18)=78264$, $a_2^\nw(2,15,4)=3$, $a_2^\nw(2,15,10)=7$, and $a_2^\nw(2,55,2)=1$.


\subsection{The Nonvanishing of \texorpdfstring{$a_2^{\nw} (4,N,k)$}{a2\^new(4,N,k)}}

From Lemma \ref{lem:a2-coeff-formula},
\begin{equation} \label{eqn:a2-new-m4}
    a_2^\nw(4,N,k) = \frac{1}{2} \lrb{ (\Tr T_4^\nw)^2 - \Tr T_{16}^\nw - 2^{k-1} \Tr T_4^\nw - 4^{k-1} \Tr T_1^\nw }.
\end{equation}
In Lemmas \ref{lem:T4-A14-bound} and \ref{lem:T1-A11-bound}, we estimated $\Tr T_4^\nw$ and $\Tr T_1^\nw$, respectively. We now estimate $\Tr{T^\nw_{16}}$.

\begin{lemma} \label{lem:T16-A116-bound}
    We have the following bound:
    \begin{equation}
        \lrabs{\frac{\Tr{T^\nw_{16}-A^\nw_1(16,N,k)}}{4^k\psi^\nw(N)}} \leq \frac{1}{8}\theta_1(N) + 94 \, \theta_3(N) + \frac{463}{16}\theta_4(N).
    \end{equation}
\end{lemma}

\begin{proof}
    We start by computing a bound for $A^\nw_2(16,N,k)$. By \eqref{eqn:A2new-explicit-bound},
    \begin{align}
        \lrabs{A^\nw_2(16,N,k)} 
        &\leq 16^{(k-1)/2}2^{\omega(N)} \cdot \sum_{t^2 < 64} 2^{\omega(64 - t^2)} \sum_n h_w \left( \frac{t^2 - 64}{n^2} \right) \psi(n) \\
        &= 376 \cdot 16^{(k-1)/2} \cdot 2^{\omega(N)} \\
        & = 94 \cdot 4^k \cdot 2^{\omega(N)}
    \end{align}
    
For $A_3^{\nw}$ we have by \eqref{eqn:A3m-formula} and Lemma \ref{lem:sigma-multiplicative},
\begin{align}
    \lrabs{A^\nw_3(16,N,k)} &= \lrabs{\frac{1}{2}\sum_{d\mid16}\min(d,16/d)^{k-1}\Sigma^\nw_{16,d}(N)} \\ 
    &= \lrabs{\Sigma^\nw_{16,1}(N) + 2^{k-1}\Sigma^\nw_{16,2}(N) + \frac12 4^{k-1}\Sigma^\nw_{16,4}(N)} \\
    & \leq 15\cdot 4^{\omega(15)} + 2^{k-1} \cdot 6 \cdot 4^{\omega(6)} + \frac{1}{8}\cdot 4^k\frac{\sqrt{N}}{\pi_2(N)^2} \\
    & = 240 + 48\cdot 2^k + \frac{1}{8}\cdot 4^{k}\frac{\sqrt{N}}{\pi_2(N)^2}.
\end{align}

And by \eqref{eqn:A4-big-O}, $\lrabs{A^\nw_4(16,N,k)} \leq \sigma_1(16) = 31$. Combining these bounds and using $k \geq 2$,

\begin{align}
    \lrabs{\frac{\Tr{T^\nw_{16}-A^\nw_1(16,N,k)}}{4^k\psi^\nw(N)}} &= \lrabs{\frac{-A^\nw_2-A^\nw_3+A^\nw_4}{4^k\psi^\nw(N)}} \\
    &\leq \frac{\lrabs{A^\nw_2}+ \lrabs{A^\nw_3} + \lrabs{A^\nw_4}}{4^k\psi^\nw(N)} \\
    &\leq \frac{94\cdot4^k\cdot2^{\omega(N)}}{4^k\psi^\nw(N)} + \frac{240+48\cdot2^k+\frac{1}{8}4^k\frac{\sqrt{N}}{\pi_2(N)^2}}{4^k\psi^\nw(N)} + \frac{31}{4^k\psi^\nw(N)} \\
    &\leq \frac{1}{8}\theta_1(N) + 94\,\theta_3(N) + \frac{463}{16}\theta_4(N),
\end{align}

as desired.
\end{proof}

We now have the tools to prove the second part of Theorem \ref{thm:m=2-4-Theorem}.

\begin{proof}

From \eqref{eqn:a2-new-m4},
\begin{align} \label{eqn:a4-new-formula}
    a_2^\nw(4,N,k) & = \frac{1}{2} \lrb{ (\Tr T_4^\nw)^2 - \Tr T_{16}^\nw - 2^{k-1} \Tr T_4^\nw - 4^{k-1} \Tr T_1^\nw }.
\end{align}
We first compute $(\Tr T_4^\nw)^2$, which contains the main term. For ease of notation, for each $m\ge1$ denote  $E_m(N, k) := \Tr T_m^\nw - A_1^\nw(m, N, k)$. Then
\begin{align}
    (\Tr T_4^\nw)^2 & = (A_1^\nw(4, N, k) + E_4(N,k))^2 \\
    & = \psi^\nw(N)^2 4^k \lrb{ \lrp{\frac{A_1^\nw(4, N, k)}{2^k\psi^\nw(N)}}^2 + 2\cdot\frac{A_1^\nw(4, N, k)}{2^k\psi^\nw(N)}\frac{E_4(N,k)}{2^k\psi^\nw(N)} + \lrp{\frac{E_4(N, k)}{2^k\psi^\nw(N)}}^2 } \\
    & = \frac{k-1}{12}\psi^\nw(N)^2 4^k \lrb{\frac{k-1}{192} + \frac{1}{2}\cdot \frac{E_4(N,k)}{2^k\psi^\nw(N)} + \frac{12}{k - 1}\lrp{\frac{E_4(N, k)}{2^k\psi^\nw(N)}}^2} \\
    & = \frac{k-1}{12}\psi^\nw(N)^2 4^k \lrb{\frac{k-1}{192} + E(N,k)},
\end{align}
Where $E(N,k)$ denotes the error terms. By Lemma \ref{lem:T4-A14-bound} and since $k\geq2$,
\begin{align} \label{eqn:quadratic-error-bound}
    |E(N,k)| & \leq \frac{1}{8}\theta_1(N) + \frac{41}{4}\theta_3(N) + \frac{19}{8}\theta_4(N) + 12\lrp{\frac{1}{4}\theta_1(N)+\frac{41}{2}\theta_3(N)+\frac{19}{4}\theta_4(N)}^2.
\end{align}
We also let
\begin{align}
    E'(N,k) :& = \frac{12}{(k-1)\psi^\nw(N)^2 4^k} \lrb{-\Tr T_{16}^\nw - \Tr T_4^\nw 2^{k-1} - \Tr T_1^\nw 4^{k-1}} \\
    & = \lrb{- \frac{12}{k-1} \cdot \frac{\Tr T_{16}^\nw}{\psi^\nw(N)^2 4^k} - \frac{6}{k-1} \cdot \frac{\Tr T_4^\nw}{\psi^\nw(N)^2 2^k} - \frac{3}{k-1} \cdot \frac{\Tr T_1^\nw}{\psi^\nw(N)^2}}, \label{eqn:Eprime-temp}
\end{align}
so that
\begin{align}
    a_2^\nw(4, N, k) = \frac{k-1}{12}\psi^\nw(N)^2 4^k \lrb{\frac{k-1}{192} + E(N,k) + E'(N,k)}.
\end{align}

Then rewriting each of the $\Tr T_m^\nw$ terms appearing in \eqref{eqn:Eprime-temp} as $A_1^\nw(m,N,k) + E_m(N,k)$, we have by \eqref{eqn:A1m-new-formula} and Lemmas \ref{lem:T4-A14-bound}, \ref{lem:T1-A11-bound}, and \ref{lem:T16-A116-bound},
\begin{align}
    |E'(N,k)| & \leq 12 \cdot \frac{16^{k/2-1}}{\psi^\nw(N)4^k} + 12\, \theta_4(N)|E_{16}(N,k)| + 6 \cdot \frac{4^{k/2-1}}{2\psi^\nw(N)2^k} + 6\,\theta_4(N)|E_4(N,k)| \\
    & \qquad + 3 \cdot \frac{1^{k/2-1}}{4\psi^\nw(N)} +  
    3\, \theta_4(N)|E_{1}(N,k)| \\
    & \leq \frac{9}{4}\theta_4(N) + 12\,\theta_4(N)\lrp{\frac{1}{8}\theta_1(N)+94\,\theta_3(N)+\frac{463}{16}\theta_4(N)} \\
    & \qquad + 6\,\theta_4(N)\lrp{\frac{1}{4}\theta_1(N)+\frac{41}{2}\theta_3(N)+\frac{19}{4}\theta_4(N)} \\
    & \qquad + 3\,\theta_4(N)\lrp{\frac{1}{2}\theta_1(N)+ \frac{7}{3}\theta_3(N)+\theta_4(N)} \\
    & = \theta_4(N) \lrp{\frac{9}{4} + \frac{9}{2}\theta_1(N) + 1258 \,\theta_3(N) + \frac{1515}{4}\theta_4(N)}. \label{eqn:linear-error-bound}
\end{align}
Combining \eqref{eqn:quadratic-error-bound} and \eqref{eqn:linear-error-bound},
\begin{align}
    |E(N,k) + E'(N,k)| & \leq \frac{1}{8}\theta_1(N) + \frac{41}{4}\theta_3(N) + \frac{37}{8}\theta_4(N) + 12\lrp{\frac{1}{4}\theta_1(N)+\frac{41}{2}\theta_3(N)+\frac{19}{4}\theta_4(N)}^2 \\
    & \qquad\qquad\qquad\qquad
    + \theta_4(N) \lrp{\frac{9}{2}\theta_1(N) + 1258\,\theta_3(N) + \frac{1515}{4}\theta_4(N)}.
\end{align}
Rearranging and using the explicit $\theta_i(N)$ bounds given in Lemma \ref{lem:computable-theta-bounds},
    \begin{align}
        |E(N,k)+ E'(N,k)| &\leq \frac{1}{8}\lrp{\frac{1}{\sqrt{N}}}+\frac{41}{4}\lrp{\frac{125.28}{N^{25/32}}}+\frac{37}{8}\lrp{\frac{12.033}{N^{63/64}}} \\ & \qquad +33\lrp{\frac{12.033}{N^{95/64}}} + \frac{3}{4}\lrp{\frac{1}{N}} + 5043\lrp{\frac{125.28^2}{N^{25/16}}} \\
        &\qquad + \frac{1299}{2}\lrp{\frac{12.033^2}{N^{63/32}}} + 123\lrp{\frac{125.28}{N^{41/32}}} + 3595\lrp{\frac{12.033 \cdot 125.28}{N^{113/64}}},
    \end{align}
which is clearly monotonically decreasing. Observe that when $N=10,\!284,\!270$, we have $|E(N,k)| \leq 0.00520829 < \frac{1}{192}$. Thus, for all $N \geq 10,\!284,\!270$ and $k \geq 2$ even, $a_2^\nw(4,N,k) > 0$. We then compute all $N < 10,\!284,\!270$ in Sage \cite{sagemath}, obtaining that $a_2^\nw(2,N,k) = 0$ for forty different pairs $(N,k)$  \cite[Table B]{ross-code}. Comparing with \cite[Tables 6.2, 6.3]{ross}, thirty-six of these pairs have $\dim S_k(\Gamma_0(N)) < 2$. The four remaining pairs for which $a_2^\nw(2,N,k)$
nontrivially vanishes are $(43,2)$, $(57,2)$, $(75,2)$, and $(205,2)$, proving the desired result.
\end{proof}

We also note from \cite[Table B]{ross-code} that $a_2^\nw(4,N,k) < 0$ for exactly $135$ pairs $(N,k)$. The minimum value achieved is $a_2^\nw(4,1,134) \approx -6.119 \times 10^{79}$.

\section{Extending to the General Character Case} \label{sec:character-case}

Now, we would like to extend our results to the case of general character.
But from \cite[Proposition 6.1]{ross}, $\dim S_k^\nw(\Gamma_0(N),\chi) = 0$ for the infinite family of triples $(N,k,\chi)$ where $2 \mid \condf(\chi)$ and $2 \parallel N/\condf(\chi)$.
This means that $a_2^\nw(m,N,k,\chi)$ trivially vanishes for infinitely many $(N,k,\chi)$.
However, if we only consider nontrivial vanishing of $a_2^\nw(m,N,k,\chi)$, then we are able to extend our result. In particular, for any given $m$, consider $N \geq 1$ coprime to $m$, $k \geq 2$, and $\chi$ a Dirichlet character modulo $N$ such that $\chi(-1) = (-1)^k$. Then we show that $a_2^\nw(m,N,k,\chi)$ nontrivially vanishes for only finitely many triples $(N,k,\chi)$, proving Theorem \ref{thm:main-theorem} for general character.


Recall the Eichler-Selberg trace formula from Lemma \ref{lemma:eichler-selberg-trace-formula},
\begin{equation}
\Tr T_m(N,k,\chi) = A_1(m,N,k,\chi) - A_2(m,N,k,\chi) - A_3(m,N,k,\chi) + A_4(m,N,k,\chi).
\end{equation}
Additionally, recall from Lemma \ref{lemma:beta} that the trace of $T_m^{\nw}(N,k,\chi)$ is given by
\begin{equation} \label{eqn:new-trace-formula-char}
    \Tr T_m^{\nw}(N,k,\chi) = \sum_{\mathfrak{f}(\chi) \mid M \mid N} \beta\lrp{\frac{N}{M}} \cdot \Tr T_m(M,k,\chi).
\end{equation}
The main difference between the general character case and the trivial character case is that this summation no longer takes the form of a Dirichlet convolution. This means that in particular, we can no longer easily write $\Tr T_m^{\nw}(N,k,\chi)$ as a linear combination of convolutions of the form $\beta*f$. However, we can still bound each of the terms in \eqref{eqn:new-trace-formula-char} separately. This rougher bound will suffice for our purposes. 

Just like in the trivial character case, we define
\begin{equation}  \label{eqn:Ai-new-char-def}
    A_i^\nw(m,N,k,\chi) := \sum_{\condf(\chi) \mid M \mid N} \beta\lrp{\frac{N}{M}} \cdot A_i(m,M,k,\chi).
\end{equation}
For a positive integer $g$, we also define 
\begin{equation}
    \psi_{g}^\nw(N) :=  \sum_{g \mid M \mid N} \beta\lrp{\frac{N}{M}} \cdot \psi(M),
\end{equation}
so that 
\begin{equation} \label{eqn:A1m-new-formula-character}
    A_1^\nw(m,N,k,\chi) = \chi(\sqrt{m})\frac{k-1}{12} m^{k/2-1} \psi_{\condf(\chi)}^\nw(N).
\end{equation}
Then in a manner similar to Lemmas \ref{lem:m-not-square} and \ref{lem:m-perfect-square}, we can determine the asymptotic behavior of $\Tr T_m^\nw(N,k,\chi)$.

\begin{lemma} \label{lem:char-tr-asymp-m-not-square}
    Let $m$ fixed not be a perfect square, and consider $N \geq 1$ coprime to $m$, $k \geq 2$, and $\chi$ a Dirichlet character modulo $N$ such that $\chi(-1) = (-1)^k$. Then
    \begin{align}
        \Tr T_m^\nw(N,k,\chi) = O\lrp{m^{k/2} 4^{\omega(N)} \sigma_0(N)}.
    \end{align}
\end{lemma}
\begin{proof}
    First, $A_1(m,N,k,\chi) = 0$ since $m$ is not a perfect square. 
    
    Second, we show that $A_2(m,N,k,\chi) = O(m^{k/2} 2^{\omega(N)})$. We have from  Huxley \cite[Page 194]{huxley} that the equation $x^2 - tx + m \equiv 0 \mod N$ has at most $2^{\omega(N)} \sqrt{\lrabs{t^2-4m}}$ solutions. Thus in the manner of \eqref{eqn:mu-bound-temp},
    \begin{align}
        \lrabs{\mu_{t,n,m}(N)} = \lrabs{\frac{\psi(N)}{\psi (N/\gcd(N,n))} \sump_{c \!\! \mod N} \chi(c)} \leq \psi(n) \cdot 2^{\omega(N)} \sqrt{\lrabs{t^2-4m}} = O(2^{\omega(N)}).
    \end{align}
    Here $t,n$ come from the the fixed value of $m$, and hence are constants with respect to the big-$O$ notation.
    Also, by \eqref{eqn:U-bounds}, $U_{k-1}(t,m) = O(m^{k/2})$. Thus
    \begin{align} 
        A_{2}(m,N,k,\chi) &=  \frac{1}{2} \sum_{t^2 < 4m} \sum_{n} U_{k-1}(t,m) h_w \lrp{ \frac{t^2-4m}{n^2} } \mu_{t,n,m}(N) \\
        &= O(m^{k/2} 2^{\omega(N)}). \label{eqn:A2-asymp-temp}
    \end{align}

    Third, we show that $A_3(m,N,k,\chi) = O(m^{k/2} 2^{\omega(N)})$. Recall from \eqref{eqn:A3m-formula},
\begin{align}
   A_3(m,N,k,\chi) &= \frac{1}{2} \sum_{d|m} \min(d,m/d)^{k-1} \mspace{-40mu} \sum_{\substack{\tau \mid N \\ (\tau,N/\tau) \mid (N/\condf(\chi),d-m/d)}} 
    \mspace{-30mu} \phi( \gcd(\tau, N / \tau )) \chi(y_{\tau}).
\end{align}
    Now, we have from \eqref{eqn:sigma-m-d-local-formula-h0}, \eqref{eqn:sigma-m-d-local-formula-hn0-1}, and \eqref{eqn:sigma-m-d-local-formula-hn0-2} that 
    \begin{align}
        \sum_{\substack{\tau \mid N \\ (\tau,N/\tau) \mid (d-m/d)}} 
    \mspace{-30mu} \phi( \gcd(\tau, N / \tau )) \ \ &\leq \ 
    \begin{cases}
         \lrabs{d-m/d} \cdot 2^{\omega(N)}  & \text{if } d-m/d \neq 0, \\
        2^{\omega(N)} \sqrt{N}  & \text{if } d-m/d=0. 
    \end{cases} 
    \end{align}
    Thus 
    \begin{align}
        \lrabs{\sum_{\substack{\tau \mid N \\ (\tau,N/\tau) \mid (N/\condf(\chi),d-m/d)}} 
    \mspace{-30mu} \phi( \gcd(\tau, N / \tau )) \chi(y_{\tau})} 
    &\leq \sum_{\substack{\tau \mid N \\ (\tau,N/\tau) \mid (d-m/d)}} 
    \mspace{-30mu} \phi( \gcd(\tau, N / \tau ))
    \\
    &= \begin{cases}
        O\lrp{2^{\omega(N)}}  & \text{if } d-m/d \neq 0, \\
        O\lrp{2^{\omega(N)} \sqrt{N}}  & \text{if } d-m/d=0. 
    \end{cases} \label{eqn:sigma-bound-cases-temp}
    \end{align}
    Note the second of these cases cannot appear here, since $m$ is not a perfect square. Thus using the fact that $\min(d,m/d)^{k-1} \leq m^{(k-1)/2}$, we have
    \begin{align}
       A_3(m,N,k,\chi) &= \frac{1}{2} \sum_{d|m} \min(d,m/d)^{k-1} \mspace{-40mu} \sum_{\substack{\tau \mid N \\ (\tau,N/\tau) \mid (N/\condf(\chi),d-m/d)}} 
        \mspace{-30mu} \phi( \gcd(\tau, N / \tau )) \chi(y_{\tau}) \\
        &= 
            O\lrp{m^{k/2} 2^{\omega(N)}},
    \end{align}
    as desired.

    Fourth, we observe from \eqref{eqn:A4m-formula} that $\lrabs{A_4(m,N,k,\chi)} \leq \sigma_1(m) = O(1)$.

    Finally, we determine the asymptotics of the $A_i^\nw(m,N,k,\chi)$. First, $A_1^\nw(m,N,k,\chi) = 0$.
    Then for $A_2^\nw(m,N,k,\chi)$, observe that there are $\leq \sigma_0(N)$ terms in the summation \eqref{eqn:Ai-new-char-def}. And by Lemma \ref{lemma:beta}, for each 
    $\beta(N/M)$ in the summation, $\lrabs{\beta(N/M)} \leq  2^{\omega(N/M)} \leq 2^{\omega(N)}$. Thus by \eqref{eqn:A2-asymp-temp},
    \begin{align} 
        \lrabs{A_2^\nw(m,N,k,\chi)} &= \lrabs{\sum_{\condf(\chi) \mid M \mid N} \beta\lrp{\frac{N}{M}} \cdot A_2(m,M,k,\chi)}  \\
        &\leq \sum_{\condf(\chi) \mid M \mid N} \lrabs{\beta\lrp{\frac{N}{M}}} \cdot \lrabs{A_2(m,M,k,\chi)} \\
        &= O\lrp{\sigma_0(N) \cdot 2^{\omega(N)} \cdot m^{k/2} 2^{\omega(N)}}.
    \end{align}
    In a similar manner, we have
    \begin{align}
        A_3^\nw(m,N,k,\chi) = 
        O\lrp{\sigma_0(N) \cdot 2^{\omega(N)} \cdot m^{k/2} 2^{\omega(N)}},  
    \end{align}
    and 
    \begin{align}
        A_4^\nw(m,N,k,\chi) = O\lrp{\sigma_0(N) \cdot 2^{\omega(N)}}.
    \end{align}
    Combining these bounds for $A_i^\nw(m,N,k,\chi)$, we obtain
    \begin{align}
        \Tr T_m^\nw(N,k,\chi) &= A_1^\nw(m,N,k,\chi) - A_2^\nw(m,N,k,\chi) - A_3^\nw(m,N,k,\chi) + A_4^\nw(m,N,k,\chi) \\
        &= O\lrp{m^{k/2} 4^{\omega(N)} \sigma_0(N)},
    \end{align}
    verifying the desired result.
\end{proof}

\begin{lemma} \label{lem:char-tr-asymp-m-square}
    Let $m \geq 1$ fixed be a perfect square, and consider $N \geq 1$ coprime to $m$, $k \geq 2$, and $\chi$ a Dirichlet character modulo $N$ such that $\chi(-1) = (-1)^k$. Then
    \begin{align}
        \Tr T_m^\nw(N,k,\chi) = \chi(\sqrt{m}) \frac{k-1}{12} m^{k/2-1} \psi_{\condf(\chi)}^\nw(N) +  O\lrp{m^{k/2} 4^{\omega(N)} \sigma_0(N) \sqrt{N}}.
    \end{align}
\end{lemma}
\begin{proof}
    We still have
    \begin{align} 
        A_2^\nw(m,N,k,\chi) &= O\lrp{\sigma_0(N) \cdot 2^{\omega(N)} \cdot m^{k/2} 2^{\omega(N)}},
    \end{align}
    and
    \begin{align}
        A_4^\nw(m,N,k,\chi) = O\lrp{\sigma_0(N) \cdot 2^{\omega(N)}},
    \end{align}
    from Lemma \ref{lem:char-tr-asymp-m-not-square}.

    For $A_3(m,N,k,\chi)$, since $m$ is a perfect square, we must consider the second case of \eqref{eqn:sigma-bound-cases-temp}. This means that we now have
    \begin{align}
       A_3(m,N,k,\chi) &= \frac{1}{2} \sum_{d|m} \min(d,m/d)^{k-1} \mspace{-40mu} \sum_{\substack{\tau \mid N \\ (\tau,N/\tau) \mid (N/\condf(\chi),d-m/d)}} 
        \mspace{-30mu} \phi( \gcd(\tau, N / \tau )) \chi(y_{\tau}) \\
        &= 
            O\lrp{m^{k/2} 2^{\omega(N)} \sqrt{N}},
    \end{align}
    and so
    \begin{align}
        A_3^\nw(m,N,k,\chi) = 
        O\lrp{\sigma_0(N) \cdot 2^{\omega(N)} \cdot m^{k/2} 2^{\omega(N)} \sqrt{N}}. 
    \end{align}

    Combining these bounds for the $A_i^\nw(m,N,k,\chi)$ and using \eqref{eqn:A1m-new-formula-character}, we obtain
    \begin{align}
        \Tr T_m^\nw(N,k,\chi) &= A_1^\nw(m,N,k,\chi) - A_2^\nw(m,N,k,\chi) - A_3^\nw(m,N,k,\chi) + A_4^\nw(m,N,k,\chi) \\
        &= \chi(\sqrt{m}) \frac{k-1}{12} m^{k/2-1} \psi_{\condf(\chi)}^\nw(N) +  O\lrp{m^{k/2} 4^{\omega(N)} \sigma_0(N) \sqrt{N}},
    \end{align}
    as desired.
\end{proof}

Next, we give a lower bound for $\psi_{\condf(\chi)}^\nw(N)$. 
In \cite[Equation (6.2)]{ross}, Ross showed that if it is not the case that $2 \mid \condf(\chi)$ and $2 \parallel N/\condf(\chi)$, then
\begin{align}
    \psi_{\condf(\chi)}^\nw(N) \geq \frac{N}{\pi_3(N)}, 
    \qquad \text{where} \qquad
     \pi_3(N) = \prod_{p \mid N} \begin{dcases}
        4 & \text{if } p=2, \\
        \lrp{1 + \frac{2}{p-2}} & \text{if } p \neq 2.
    \end{dcases} \label{eqn:psi-f-new-lower-bound}
\end{align}

We now have the tools to prove Theorem \ref{thm:main-theorem} for general character.

\begin{proof}
    Since we are only considering $a_2^\nw(m,N,k,\chi)$ nontrivially vanishing, we can assume it is not the case that $2 \mid \condf(\chi)$ and $2 \parallel N/\condf(\chi)$; otherwise we would have $\dim S_k(\Gamma_0(N),\chi) = 0$.
    
    Now, let $T_m^\nw$ denote $T_m^{\nw}(N,k,\chi)$, and let $g = \condf(\chi)$. Then recall from Lemma \ref{lem:a2-coeff-formula} that
    \begin{align}   \label{eqn:a2-coeff-formula-temp}
        a_2^{\nw}(m,N,k,\chi) = \frac{1}{2} \lrb{ \lrp{\Tr T_m^{\nw}}^2 - \sum_{d \mid m} \chi(d) d^{k-1} \Tr T_{m^2/d^2}^{\nw} }.
    \end{align}
    Then applying Lemma \ref{lem:char-tr-asymp-m-square} to the summation in \eqref{eqn:a2-coeff-formula-temp}, we obtain
    \begin{align}
        \sum_{d \mid m} \chi(d) d^{k-1} \Tr T_{m^2/d^2}^{\nw} 
        &= \sum_{d \mid m}  \chi(d) d^{k-1} \Bigg[ \chi\lrp{\sqrt{\frac{m^2}{d^2}}} \frac{k-1}{12} \lrp{\frac{m^2}{d^2}}^{k/2-1} \psi_{g}^\nw(N)  \\
        & \qquad\qquad\qquad\qquad +  O\lrp{\lrp{\frac{m^2}{d^2}}^{k/2} 4^{\omega(N)} \sigma_0(N) \sqrt{N}} \Bigg] \\
        &= \sum_{d \mid m} \lrb{\chi(m) d\, \frac{k-1}{12}  m^{k-2}  \psi_{g}^\nw(N) \,+\, O\lrp{m^k  4^{\omega(N)} \sigma_0(N) \sqrt{N}}} \\
        &= \chi(m) \sigma_1(m)  m^{k-2}  \psi_{g}^\nw(N) \frac{k-1}{12}  \,+\, O\lrp{m^k 4^{\omega(N)} \sigma_0(N) \sqrt{N}}.
        \label{eqn:a2sum-bound-temp}
    \end{align}
    Now, if $m$ is not a perfect square, then we apply Lemma \ref{lem:char-tr-asymp-m-not-square} to the $(\Tr T_m^\nw)^2$ term from \eqref{eqn:a2-coeff-formula-temp} and obtain
    \begin{align} \label{eqn:a2Tm2-bound-temp}
        (\Tr T_m^\nw)^2 &= O\lrp{m^k 16^{\omega(N)} \sigma_0(N)^2}. 
    \end{align}

    Then combining \eqref{eqn:a2sum-bound-temp} and \eqref{eqn:a2Tm2-bound-temp}, we obtain 
    \begin{align}  
        a_2^{\nw}(m,N,k,\chi)  
        &= \frac{1}{2} \lrb{ \lrp{\Tr T_m^{\nw}}^2 - \sum_{d \mid m} \chi(d) d^{k-1} \Tr T_{m^2/d^2}^{\nw} } \\
        &= \frac12 \Bigg[  
        O\lrp{m^k 16^{\omega(N)} \sigma_0(N)^2} - \chi(m) \sigma_1(m)  m^{k-2}  \psi_{g}^\nw(N) \frac{k-1}{12}  \\
        & \mspace{290mu} - O\lrp{m^k 4^{\omega(N)} \sigma_0(N) \sqrt{N}}
        \Bigg] \\
        &= \frac{\chi(m) \sigma_1(m) m^{k-2} \psi_{g}^\nw(N) }{2} \lrb{ - \frac{k-1}{12} +  O\lrp{\frac{16^{\omega(N)} \sigma_0(N)^2 \sqrt{N}}{\psi_g^\nw(N)}}  }.
        \label{eqn:a2new-aymp-temp}
    \end{align}
    Then recall from \eqref{eqn:psi-f-new-lower-bound} that $\psi_{g}^\nw(N) \geq \frac{N}{\pi_3(N)}$. Additionally, $\frac13 \pi_3(N) \leq 2^{\omega(N)} \leq \sigma_0(N) = O(N^\varepsilon)$ for any $\varepsilon > 0$ \cite[Sections~18.1,~22.13]{hardy-wright}. Thus the $O(\cdot)$ error term in \eqref{eqn:a2new-aymp-temp} is $O \lrp{N^{-1/2 + \varepsilon}}$ and hence $\longrightarrow 0$ as $N \longrightarrow \infty$. So by a similar argument as in Proposition \ref{prop:nsquare}, $a_2^{\nw}(m,N,k,\chi)$ nontrivially vanishes for only finitely many triples $(N,k,\chi)$.

    If $m$ is a perfect square, then we have from Lemma \ref{lem:char-tr-asymp-m-square}, 
    \begin{align} \label{eqn:a2Tm2-sq-bound-temp}
         (\Tr T_m^\nw)^2 &= \lrp{\chi(\sqrt{m}) \frac{k-1}{12} m^{k/2-1} \psi_{g}^\nw(N) +  O\lrp{m^{k/2} 4^{\omega(N)} \sigma_0(N) \sqrt{N}}}^2 \\
         &=\chi(m) m^{k-2} \psi_g^\nw(N)^2 \frac{(k-1)^2}{144} + O\lrp{(k-1)m^k 4^{\omega(N)} \sigma_0(N) \sqrt{N} \psi_g^\nw(N)} 
         \label{eqn:a2Tm2-bound-temp2}
    \end{align}
    Here, we used the fact that $4^{\omega(N)} \sigma_0(N) \sqrt{N} = O\lrp{\psi_g^\nw(N)}$, as noted above.

    Then combining \eqref{eqn:a2sum-bound-temp} and \eqref{eqn:a2Tm2-bound-temp2}, we obtain 
    \begin{align}  
        a_2^{\nw}(m,N,k,\chi)  
        &= \frac{1}{2} \lrb{ \lrp{\Tr T_m^{\nw}}^2 - \sum_{d \mid m} \chi(d) d^{k-1} \Tr T_{m^2/d^2}^{\nw} } \\
        &= \frac12 \lrb{
        \chi(m) m^{k-2} \psi_g^\nw(N)^2 \frac{(k-1)^2}{144} + O\lrp{(k-1)m^k 4^{\omega(N)} \sigma_0(N) \sqrt{N} \psi_g^\nw(N)}} 
         \\
        &= \frac{\chi(m) (k-1) m^{k-2} \psi_{g}^\nw(N)^2 }{2} \lrb{ \frac{k-1}{144} +  
        O\lrp{ \frac{4^{\omega(N)} \sigma_0(N) \sqrt{N}}{\psi_g^\nw(N)} }} 
        \label{eqn:a2new-aymp-temp-sq}
    \end{align}
    Again, we have the $O(\cdot)$ error term $\longrightarrow 0$, so in this case as well, $a_2^{\nw}(m,N,k,\chi)$ nontrivially vanishes for only finitely many triples $(N,k,\chi)$.     
\end{proof}

\section{Discussion} \label{sec:discussion}
In this section, we discuss some motivation for the study of $a_2^\nw(m,N,k)$, as well as potential future work. 

As noted previously, we observe that the second coefficients $a_2(m,N,k)$ and $a_2^\nw(m,N,k)$ seem to be easier to study than the first coefficients $a_1(m,N,k)$ and $a_1^\nw(m,N,k)$. This is due in a large part to the fact that we can estimate the growth of $a_2(m,N,k)$ and $a_2^\nw(m,N,k)$ (e.g. \eqref{eqn:a2new-final-formula-temp}, \eqref{eqn:a2new-final-formula-temp2}), but $a_1(m,N,k) = \Tr T_m(N,k)$ and $a_1^\nw(m,N,k) = \Tr T_m^\nw(N,k)$ do not seem to follow any sort of asymptotic behavior.
This observation has several potential applications.  
Traces of Hecke operators have been used in the past to prove several interesting results about modular forms, and in certain of these scenarios, the second coefficient can be used instead of the trace to obtain stronger results. For example, in \cite{vilardi-xue}, Vilardi and Xue used the non-repetition of $\Tr T_2(1,k)$ in order to prove that, assuming Maeda's conjecture, Hecke eigenforms can be distinguished by their $2$nd Fourier coefficient. More recently, Clayton et. al. \cite{nonrep} instead studied the non-repetition of $a_2(m,N,k)$. They were able to use these non-repetition results for $a_2(4,1,k)$ in order to prove that, assuming Maeda's conjecure, Hecke eigenforms can also be distinguished by their $4$th Fourier coefficient. For any given $m \ge 2$, they also provided a general strategy to prove the same result for the $m$-th Fourier coefficient. 

Additionally, recall that the classical Lehmer conjecture predicts the non-vanishing of $\tau(m) = \Tr T_m(1,12)$. In order to generalize this conjecture to higher levels and weights, Rouse proposed the ``generalized Lehmer conjecture" \cite[Conjecture 1.5]{rouse}, predicting that for any non-square $m$, $\Tr T_m(N,k) \neq 0$ for all $N$ coprime to $m$ and even $k \ge 12$, $k \ne 14$.  
We note that since we also have $\tau(m) = \Tr T_m^\nw(1,12)$, one could also attempt to generalize the Lehmer conjecture to newspaces of higher level and weight.
However, this generalization turns out to not be true; we were able to find several families of $m,N,k$ for which $\Tr T_m^\nw(N,k)$ vanishes \cite{ross-code}. In future work, we plan to investigate various cases in which this occurs. At a minimum, we propose the following conjecture, based on our numerical computations.
\begin{conjecture} \label{conj:gen-Lehmer-newspace}
    Let $m \ge 2$ be a fixed non-square. Then there exist $N_1$ and $N_2$ coprime to $m$ such that $\Tr T_m^\nw(m,N_1,k) = 0$ for all even $k \ge 2$ and $\Tr T_m^\nw(m,N_2,k) \ne 0$ for all even $k \ge 2$.
\end{conjecture}

Like the generalized Lehmer conjecture from Rouse, determining precisely when $\Tr T_m^\nw(N,k)$ vanishes seems to be a rather difficult problem. However, as the content of this paper shows, we are actually able to obtain results if we consider the second coefficient $a_2^\nw(m,N,k)$ instead of the first coefficient $a_1^\nw(m,N,k)$. (And in fact we can even generalize slightly to general character.) One of our goals is that further study of these other coefficients $a_j(m,N,k)$ and $a_j^\nw(m,N,k)$ will lead to ideas and strategies to answer questions about the behavior of the first coefficient $a_1(m,N,k) = \Tr T_m(N,k)$ and $a_1^\nw(m,N,k) = \Tr T_m^\nw(N,k)$. Similarly, we are also interested in answering questions about the behavior of the last coefficient $a_{\dim S_k(\Gamma_0(N))}(m,N,k)$ and $a_{\dim S_k^\nw(\Gamma_0(N))}^\nw(m,N,k)$.
This is just the determinant of $T_m(N,k)$ and $T_m^\nw(N,k)$, so any results concerning the behavior of this last coefficient would be very valuable in understanding the Hecke operators in general.
Much more work and investigation remains to be completed in this direction.


Now, the results of \cite{ross-xue} show that the second coefficient $a_2(m,N,k)$ (over $S_k(\Gamma_0(N))$) is biased to be positive for square $m$ and negative for non-square $m$. Moreover, this paper establishes that the same bias holds when one restricts to the newspace $S_k^\nw(\Gamma_0(N))$. One might then ask if the second coefficient exhibits a similar behavior for other restrictions. For example, these two results would lead one to intuitively expect the same result on the old space $S_k^{\text{old}}(\Gamma_0(N))$. However, this needs to be proven and made more precise. For another example, one could also ask the same question for the restrictions to the subspaces determined by Atkin-Lehner sign pattern $S_k^{\nw,\varepsilon_M}(\Gamma_0(N))$ \cite{kimball-martin}.

Lastly, we note that the second coefficient is closely related to the second moment, and hence the quadratic mean, of the eigenvalues. Using the techniques developed to study $a_2(m,N,k)$ and $a_2^\nw(m,N,k)$, we were also able to determine the average size of the eigenvalues of the Hecke operators (measured via the quadratic mean) \cite{avg-size}. Using a similar strategy, further work can be done to similarly compute the average size of the eigenvalues over $S_k^{\nw,\varepsilon_M}(\Gamma_0(N))$ and other subspaces. Additionally, further investigation could be done to compute other moments of the eigenvalues of the Hecke operators $T_m(N,k)$ and $T_m^\nw(N,k)$.

\section*{Acknowledgements}
This research was supported by NSA MSP grant H98230-24-1-0033. 

\bibliographystyle{plain}
\bibliography{bibliography.bib}

\end{document}